\documentclass[12pt]{amsart}

\usepackage[utf8]{inputenc}
\usepackage{amsmath}
\usepackage{amssymb}
\usepackage{amsthm}
\usepackage{array}
\usepackage{epsfig}
\usepackage{graphicx}
\usepackage{xy}
\usepackage{setspace}

\usepackage{bm}
\usepackage{psfrag}
\usepackage{enumerate}
\usepackage[dvipsnames]{xcolor}
\usepackage{url}
\usepackage{tikz}
\usepackage{blkarray}

\newcommand{\fig}{\begin{figure}[htbp]\centering}

\newcommand{\efig}{\end{figure}} 

\newcommand\nothing{\varnothing}

\newcommand\HH{\widetilde{H}}
\newcommand\bb{\mathbf{b}}

\newcommand\kk{\Bbbk}

\newcommand\xx{{\mathbf x}}

\renewcommand\aa{{\mathbf a}}

    \newtheorem{thm}{Theorem}[section]
\newtheorem{lemma}[thm]{Lemma}

\newtheorem{cor}[thm]{Corollary}
\newtheorem{prop}[thm]{Proposition}

\theoremstyle{definition}
\newtheorem{example}[thm]{Example}
\newtheorem{remark}[thm]{Remark}
\newtheorem{defn}[thm]{Definition}

\newcommand\st{\mathit{ST\hspace{-.2ex}}}
\newcommand\bs{\backslash}
\newcommand\dis{\displaystyle}
\newcommand\foot{\footnotesize}
\newcommand\fillbar{\mkern-6mu\cleaders\hbox{$\mkern-2mu \mathord- \mkern-2mu$}\hfill
	\mkern-6mu \mathord-}
\newcommand\mko[1]{\makebox[0pt][c]{$#1$}}
\newcommand\edgeup[1]{/\makebox[0pt][l]{\raisebox{-.4ex}{$\!\scriptstyle#1$}}}
\newcommand\edgedown[1]{\bs\makebox[0pt][l]{\raisebox{.7ex}{$\!\scriptstyle#1$}}}

\newcommand\edgehoriz[1]{\raisebox{0pt}[0pt][0pt]{$\stackrel{#1}{\hbox{ --- }}$}}
\newcommand\monomialmatrix[3]{{
\begin{array}{@{}r@{\:}r@{}c@{}l@{}}
  \begin{array}{@{}c@{}}		
	\begin{array}{@{}r@{}}
	\\
	#1
	\end{array}\!
  \end{array}						
&
  \begin{array}{@{}c@{}}		
	\begin{array}{@{}l@{}}\\				
	\end{array}						
	\\							
	\left[\begin{array}{@{}l@{}}				
	#3							
	\end{array}\!						
	\right.							
  \end{array}							
&
  #2					
&
  \begin{array}{@{}c@{}}		
	\begin{array}{@{}l@{}}\\				
	\end{array}						
	\\							
	\left.\!\begin{array}{@{}l@{}}				
	#3							
	\end{array}						
	\right]							
  \end{array}							
\end{array}
}}
\newcommand\mkl[1]{\makebox[0pt][l]{$#1$}}

\newcommand{\casezero}{\raisebox{-0.5em}{\begin{tikzpicture}[scale=0.3, every node/.style={scale=0.7}]
    \draw[thick] (0.3,1) -- (1.7,1);
    \draw[thick] (0.3,1) -- (1,0);
    \draw[thick] (1,0) -- (1.7,1);
    \filldraw (0.3,1) circle (3pt)  node[left] {$j$};
    \filldraw (1.7,1) circle (3pt) node[right] {$i$};
    \filldraw (1,0) circle (3pt) node[right] {$k$};
\end{tikzpicture}}}
\newcommand{\caseone}{\raisebox{-0.5em}{\begin{tikzpicture}[scale=0.3, every node/.style={scale=0.7}]
    \draw[thick] (0.3,1) -- (1.7,1);
    \draw[gray!80, densely dotted] (0.3,1) -- (1,0);
    \draw[thick] (1,0) -- (1.7,1);
    \filldraw (0.3,1) circle (3pt)  node[left] {$j$};
    \filldraw (1.7,1) circle (3pt) node[right] {$i$};
    \filldraw (1,0) circle (3pt) node[right] {$k$};
\end{tikzpicture}}
}
\newcommand{\casetwo}{\raisebox{-0.5em}{\begin{tikzpicture}[scale=0.3, every node/.style={scale=0.7}]
    \draw[thick] (0.3,1) -- (1.7,1);
    \draw[gray!80, densely dotted] (0.3,1) -- (1,0);
    \draw[gray!80, densely dotted] (1,0) -- (1.7,1);
    \filldraw (0.3,1) circle (3pt)  node[left] {$j$};
    \filldraw (1.7,1) circle (3pt) node[right] {$i$};
    \filldraw (1,0) circle (3pt) node[right] {$k$};
\end{tikzpicture}}}
\newcommand{\casefour}{\raisebox{-0.5em}{\begin{tikzpicture}[scale=0.3, every node/.style={scale=0.7}]
    \draw[gray!80, densely dotted] (0.3,1) -- (1.7,1);
    \draw[gray!80, densely dotted] (0.3,1) -- (1,0);
    \draw[gray!80, densely dotted] (1,0) -- (1.7,1);
    \filldraw (0.3,1) circle (3pt)  node[left] {$j$};
    \filldraw (1.7,1) circle (3pt) node[right] {$i$};
    \filldraw (1,0) circle (3pt) node[right] {$k$};
\end{tikzpicture}}}
\newcommand{\casethree}{\raisebox{-0.5em}{\begin{tikzpicture}[scale=0.3, every node/.style={scale=0.7}]
    \draw[gray!80, densely dotted] (0.3,1) -- (1.7,1);
    \draw[gray!80, densely dotted] (0.3,1) -- (1,0);
    \draw[gray!80, densely dotted] (1,0) -- (1.7,1);
    \filldraw (0.3,1) circle (3pt)  node[left] {$j$};
    \filldraw (1.7,1) circle (3pt) node[right] {$i$};
    \filldraw[gray!60] (1,0) circle (3pt);
\end{tikzpicture}}}
\newcommand{\casethreeprime}{\raisebox{-0.5em}{\begin{tikzpicture}[scale=0.3, every node/.style={scale=0.7}]
    \draw[gray!80, densely dotted] (0.3,1) -- (1.7,1);
    \draw[gray!80, densely dotted] (0.3,1) -- (1,0);
    \draw[gray!80, densely dotted] (1,0) -- (1.7,1);
    \filldraw (0.3,1) circle (3pt)  node[left] {$j$};
    \filldraw[gray!60] (1.7,1) circle (3pt);
    \filldraw (1,0) circle (3pt) node[right] {$k$};
\end{tikzpicture}}}
\newcommand{\casefive}{\raisebox{-0.5em}{\begin{tikzpicture}[scale=0.3, every node/.style={scale=0.7}]
    \draw[thick] (0.3,1) -- (1.7,1);
    \draw[gray!80, densely dotted] (0.3,1) -- (1,0);
    \draw[gray!80, densely dotted] (1,0) -- (1.7,1);
    \filldraw (0.3,1) circle (3pt)  node[left] {$j$};
    \filldraw (1.7,1) circle (3pt) node[right] {$i$};
    \filldraw[gray!50] (1,0) circle (3pt);
\end{tikzpicture}}}

\usepackage{relsize}

\title{Enumerative and planar combinatorics of trivariate monomial resolutions}
\author{Erika Ordog}
\date{January 2021}

\begin{document}

\begin{abstract}
    The canonical sylvan resolution is a resolution of an arbitrary monomial ideal over a polynomial ring that is minimal and has an explicit combinatorial formula for the differential.
    The differential is a weighted sum over lattice paths of weights of chain-link fences, which are sequences of faces that are linked to each other via higher-dimensional analogues of spanning trees.
    Along a lattice path in the three-variable case, these weights can be condensed to a single weight contributing to the combinatorial formula for the differential that bypasses any computation of chain-link fences.
    The main results in this paper express the sylvan matrix entries for monomial ideals in three variables as a sum over lattice paths of simpler weights that depend only on the number of specific Koszul simplicial complexes that lie along the corresponding lattice path.
    Certain entries have numerators equal to the number of lattice paths in $\mathbb{N}^2$ that follow specific restrictions.
\end{abstract}

\maketitle

\begin{singlespace}
\section{Introduction}
Resolutions of monomial ideals have been studied extensively since the introduction of the problem by Kaplansky in the early 1960s. 
There have since been dozens of constructions of monomial resolutions with combinatorial descriptions. 
In three variables, monomial ideals enjoy special properties that allow for nice combinatorial resolutions, and genericity can be determined by looking at the structure of a single matrix in a given resolution.
Staircase diagrams provide all necessary information for a combinatorial resolution for trivariate monomial ideals.
Constructions of resolutions using staircase diagrams include resolutions by planar graph, which are supported on planar graphs embedded in the staircase diagram \cite{miller-Planar2002}, and Buchberger resolutions for strongly generic monomial ideals, also given by embeddings of planar graphs in the staircase diagrams (see \cite{gm88} and \cite[Theorem 3.11]{cca}).
When a trivariate monomial ideal $I$ is primary to the homogeneous maximal ideal, the structure of the last matrix in a resolution for $I$ determines whether $I$ is generic \cite{painter15}.

Even in the three-variable case, there had not previously been a universal, canonical, minimal free resolution of monomial ideals with a closed-form, combinatorial description of the differential. 
The resolutions by planar graph in \cite{miller-Planar2002} are noncanonical, and the Buchberger resolutions only work for the case of strongly generic monomial ideals.

The canonical sylvan resolution in \cite{emo20} is the first construction of a free resolution for monomial ideals over the polynomial ring $S = \kk[x_1, \ldots, x_n]$ in any number of variables that is universal, canonical, and minimal with a closed-form combinatorial formula for the differential.
The differential mapping from the part of the free module $F_i$ generated in $\mathbb{N}^n$-degree $\bb$ to the part of $F_{i-1}$ generated in degree $\aa$ relies on combinatorics of Koszul simplicial complexes along lattice paths from $\bb$ to $\aa$ in $\mathbb{N}^n$.
For a monomial ideal $I$, the Koszul simplicial complex of $I$ in $\mathbb{N}^n$-degree $\bb$ is the set $K^{\bb}I = \{ \text{squarefree } \tau \mid \xx^{\bb - \tau} \in I \}$, and by a result of Hochster \cite{hochster77} (see \cite{cca}, Theorem 1.34), the free modules in a minimal free resolution of $I$ can be written $F_i = \bigoplus_{\bb \in \mathbb{N}^n} \widetilde{H}_{i-1}(K^{\bb}I;\kk)$.
The canonical sylvan homomorphisms between the free $S$-modules
\[ \bigoplus\limits_{\mathbf{a} \in \mathbb{N}^n} \widetilde{H}_{i-1} (K^{\mathbf{a}}I;\kk) \otimes_{\kk} S(-\mathbf{a}) \leftarrow \bigoplus\limits_{\mathbf{b} \in \mathbb{N}^n} \widetilde{H}_i (K^{\bb}I; \kk) \otimes_{\kk} S(-\bb) \]
are induced by homomorphisms on the chains $\widetilde{C}_{i-1} (K^{\mathbf{a}}I;\kk) \leftarrow \widetilde{C}_{i} (K^{\mathbf{b}}I; \kk)$.
The differential is a weighted sum over lattice paths of weights of chain-link fences, which are sequences of faces that are linked to each other via higher-dimensional analogues of spanning trees in Koszul simplicial complexes indexed by the lattice.
The resolution results from using the Moore--Penrose pseudoinverse as the splitting for a Wall complex \cite{eag90}, and the combinatorial formula for the differential relies on a combinatorial formula for the Moore--Penrose pseudoinverse, which follows as a corollary of a result of Berg \cite[Theorem 1]{berg86} (see \cite[Theorem 5.7]{emo20}). 
The differential is expressed via sylvan matrices, whose rows and columns are indexed by faces of Koszul complexes in appropriate degrees.

In addition, there are simpler, but noncanonical, sylvan resolutions that arise via choices of higher-dimensional spanning trees at points along the lattice paths in $\mathbb{N}^n$. 
This presents the problem of classifying existing monomial resolutions: which ones can be represented as noncanonical sylvan resolutions, with certain choices of spanning trees? 
This paper shows that the resolutions by planar graph in \cite{miller-Planar2002} are indeed sylvan.

This paper shows that for the canonical sylvan resolution, the weights of chain-link fences along lattice paths can be condensed to a single weight that requires no computations of chain-link fences.
The sylvan matrix entries can then be expressed as a sum over lattice paths of weights that depend only on the number of certain Koszul simplicial complexes at points that lie along the lattice path.
These formulas are given in Theorems~\ref{compendium0to1} and \ref{threevariabletheorem}.
The computation of some of these entries can be reduced to counting lattice paths in $\mathbb{N}^2$ with certain restrictions; see Corollary~\ref{gridcounttheorem}.
Since the denominators of these values are only divisible by primes $2$ and $3$, these results can be applied to polynomial rings over any field $\kk$ such that $\text{char}(\kk) \neq 2,3$.
Examples are given in Section 5.

\subsection*{Acknowledgments}
I am grateful to Ezra Miller for helpful comments on this paper.

\section{Preliminaries}
Unlike the homology vector spaces $\widetilde{H}_i (K^{\bb}I;\kk)$, the chains $\widetilde{C}_i (K^{\bb}I;\kk)$ have canonical bases given by the $i$-faces of $K^{\bb}I$.
Since writing down a matrix for the homomorphisms of the appropriate dimension involves a choice of basis, the resolution is instead written down using sylvan matrices, where the rows and columns are indexed by the bases for the reduced chains.

\begin{defn}\label{sylvanmatricesdefn}
The \textit{sylvan matrix} $D^{\aa \bb}$ for the map $\widetilde{H}_{i-1} K^{\aa}I \leftarrow \widetilde{H}_i K^{\bb}I$ of subquotients of the chain groups is the matrix  
$$%
\HH_{i-1} K^{\aa} \!\otimes\! \langle \xx^{\aa} \rangle
\xleftarrow{
\monomialmatrix
  {\sigma_1 \\ \sigma_2\\ \vdots \\\sigma_s}
  {\begin{array}{@{}l@{\ }c@{\ }c@{\ }r@{}}
    \quad\, \tau_1 &\ \ \tau_2 &\, \, \cdots &\ \, \tau_r
    \\
       &&&
       \\&&&
       \\&&&
       \\&&&
   \end{array}}
  {\\\\\\\\}
\!\!\!}
{\HH_i K^{\bb} \!\otimes\! \langle \xx^{\bb} \rangle}
$$
where $\tau_{\ell}$ is an $i$-face in $K^{\bb}I$, $\sigma_{\ell}$ is an $(i-1)$-face in $K^{\aa}I$.
The entry $D_{\sigma \tau}^{\aa \bb}$ is the coefficient of $\sigma \otimes \xx^{\aa} \cdot \xx^{\bb - \aa}$ in $D(\tau \otimes \xx^{\bb})$, where $\widetilde{H}_{i-1}K^{\aa}I \otimes \langle \xx^{\aa} \rangle \xleftarrow{D} \widetilde{H}_i K^{\bb}I \otimes \langle \xx^{\bb} \rangle$ is the differential of the sylvan resolution.
\end{defn}

The combinatorial formula for the differential in the sylvan resolution is given as a sum of weights of chain-link fences, which are sequences of faces that are related to each other via shrubberies, stake sets, and hedges.

\begin{defn}\label{d:hedge}
Let $K_i$ be the set of $i$-faces of a simplicial complex $K$.
\begin{enumerate}
\item A \textit{shrubbery} in dimension $i$ is a subset $T_i \subseteq K_i$ whose set $\partial T_i = \{ \partial \tau \mid \tau \in T_i \}$ of boundaries is a basis for the boundaries $\widetilde{B}_{i-1}(K;\kk) = \partial(\widetilde{C}_i(K);\kk)$.
\item A \textit{stake set} in dimension $i-1$ is a set of $(i-1)$-faces $S_{i-1} \subseteq K_{i-1}$ whose complement $\bar{S}_{i-1}$ gives a basis for $\widetilde{C}_{i-1}(K;\kk)/ \widetilde{B}_{i-1}(K;\kk)$.
\item A \textit{hedge} $ST_i = (S_{i-1}, T_i)$ in $K$ of dimension $i$ is a choice of shrubbery $T_i \subseteq K_i$ and stake set $S_{i-1} \subseteq K_{i-1}$.
\end{enumerate}
\end{defn}

Note that a shrubbery is often referred to as a spanning tree or a spanning forest; see \cite{dkm09} and \cite{cck17}.

\begin{example}
When $K$ is one-dimensional, each $T_1$ is a spanning forest in the usual sense for a graph, meaning it is a spanning tree on each connected component.
In the basis for the chains $\widetilde{C}_i (K;\kk)$ consisting of $i$-faces, the faces in a stake set $S_i$ can be replaced with boundaries. 
Thus the complement $\bar{S}_0$ consists of a single vertex (a root) for each connected component of $K$.

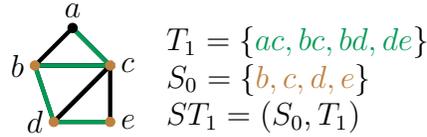
\begin{center}
\begin{tikzpicture}[scale=0.5]
\draw[ultra thick] (0,5) -- (1,4) -- (-1,4) -- (0,5);
\draw[ultra thick] (-1,4) -- (-0.5,2.5) -- (1,4) -- (-1,4);
\draw[ultra thick] (-0.5,2.5) -- (1,2.5) -- (1,4) -- (-0.5,2.5);
\draw[ultra thick, ForestGreen] (0,5) -- (1,4) -- (-1,4) -- (-0.5,2.5) -- (1,2.5);
\filldraw (0,5) circle (3.5pt);
\node[above] at (0,5) {$a$};
\filldraw[brown] (-1,4) circle (3.5pt);
\node[left] at (-1,4) {$b$};
\filldraw[brown] (1,4) circle (3.5pt);
\node [right] at (1,4) {$c$};
\filldraw[brown] (-0.5,2.5) circle (3.5pt);
\node[left] at (-0.5,2.5) {$d$};
\filldraw[brown] (1,2.5) circle (3.5pt);
\node[right] at (1,2.5) {$e$};
\end{tikzpicture}
\begin{tikzpicture}[scale=0.5]
\node[right] at (-8,4.5) {$T_1 = \{ {\color{ForestGreen} ac,bc,bd,de} \} $};
\node[right] at (-8,3.5) {$S_0 = \{ {\color{brown} b,c,d,e} \}$};
\node[right] at (-8,2.5) {$ST_1 = (S_0, T_1)$};
\end{tikzpicture}
\end{center}
\end{example}

\begin{defn}
Fix a hedge $ST_i$, a stake set $S_i \subseteq K_i$, an $i$-face $\tau \in K_i$, and a stake $\sigma \in S_{i-1}$.
\begin{enumerate}
    \item The \textit{circuit} of $\tau$ is the unique cycle $\zeta_{T_i}(\tau) = \tau - t$, where $t \in \kk \{ T_{i} \}$.
    \item $\tau$ is \textit{cycle-linked} to $\tau'$ if $\tau'$ has nonzero coefficient in $\zeta_{T_i}(\tau)$.
    \item The \textit{shrub} of $\sigma$ is the unique chain $s(\sigma) \in \kk \{ T_i \}$ whose boundary has coefficient $1$ on $\sigma$ and $0$ on all other stakes in $S_{i-1}$.
    \item $\sigma$ is \textit{chain-linked} to $\tau'$ if $\tau'$ has nonzero coefficient in the shrub $s(\sigma)$.
    \item The \textit{hedge rim} of $\tau$ is the unique chain $r(\tau) \in \kk \{ K_i \setminus S_i \}$ such that $\tau - r(\tau)$ is a boundary in $K$.
    \item $\tau$ is \textit{boundary-linked} to $\tau'$ if $\tau'$ has nonzero coefficient in the hedge rim $r(\tau)$.
\end{enumerate}
\end{defn}

\begin{example}
Consider the simplicial complex of dimension $1$ given below.
Let $ST_1$ be the hedge $(S_0, T_1)$, where $S_0 = \{ b, c,d,e\}$ and $T_1 = \{ ac,bc,bd,de \}$.
Then the circuit of the edge $cd$ is $\zeta_{T_1} (cd) = cd + bc - bd$, the shrub of the vertex $b$ is $s(b) = bc - ac$, and the hedge rim of $b$ is $r(b) = a$.
This means that $cd$ is cycle-linked to $cd, bc$, and $bd$.
The vertex $b$ is chain-linked to the edges $bc$ and $ac$, and it is boundary-linked to the vertex $a$.
\end{example}

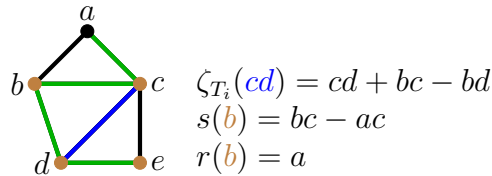
\begin{center}
\begin{tikzpicture}[scale=0.5]
\draw[ultra thick] (0,5) -- (1,4) -- (-1,4) -- (0,5);
\draw[ultra thick] (-1,4) -- (-0.5,2.5) -- (1,4) -- (-1,4);
\draw[ultra thick] (-0.5,2.5) -- (1,2.5) -- (1,4) -- (-0.5,2.5);
\draw[ultra thick, green!70!black] (0,5) -- (1,4) -- (-1,4) -- (-0.5,2.5) -- (1,2.5);
\draw[ultra thick, blue] (-0.5,2.5) -- (1,4);
\filldraw (0,5) circle (3.5pt);
\node[above] at (0,5) {$a$};
\filldraw[brown] (-1,4) circle (3.5pt);
\node[left] at (-1,4) {$b$};
\filldraw[brown] (1,4) circle (3.5pt);
\node [right] at (1,4) {$c$};
\filldraw[brown] (-0.5,2.5) circle (3.5pt);
\node[left] at (-0.5,2.5) {$d$};
\filldraw[brown] (1,2.5) circle (3.5pt);
\node[right] at (1,2.5) {$e$};
\end{tikzpicture}
\begin{tikzpicture}[scale=0.5]
\node[right] at (-8,4.5) {$\zeta_{T_i}({\color{blue} cd}) = cd +bc - bd $};
\node[right] at (-8,3.5) {$s({\color{brown} b}) = bc - ac$};
\node[right] at (-8,2.5) {$r({\color{brown} b}) = a$};
\end{tikzpicture}
\end{center}

\begin{defn}
Let $\Lambda (\aa, \bb)$ be the set of all saturated, decreasing lattice paths from $\bb$ to $\aa$ in $\mathbb{N}^n$.
Given a lattice path $\lambda = (\aa= \bb_j, \bb_{j-1}, \ldots, \bb_1, \bb_0=\aa) \in \Lambda(\aa,\bb)$, a \textit{hedgerow} $ST_i^{\lambda}$ on $\lambda$ consists of a stake set $S_i^{\bb} \subseteq K^{\bb}I$, a hedge $ST_i^{\mathbf{c}} \subseteq K^{\mathbf{c}}I$ for each degree $\mathbf{c} = \bb_1, \ldots, \bb_{j-1}$, and a shrubbery $T_{i-1}^{\aa} \subseteq K_{i-1}^{\aa}I$. 
\end{defn}

\begin{defn}
Given a lattice path $\lambda \in \Lambda(\aa,\bb)$,  a \textit{chain-link
fence} $\phi$ from an $i$-face $\tau$ to an $(i-1)$-face
$\sigma$ along $\lambda$ consists of a hedgerow $ST_i^{\lambda}$ and a
sequence of faces
$$%
\begin{array}{*{15}{@{}c@{}}}
\\[-3.2ex]
&&\!\!\tau_{j-1}\!\!\!&&&&\ \cdots\ &&&&\tau_1&&&&\tau_0\hbox{ --- }\tau\\
&/&&\bs&&/&&\bs&&/&&\bs&&/&\\
\sigma\hbox{ --- }\sigma_j&&&&\!\!\sigma_{j-1}\!\!\!&&&&\sigma_2&&&&\sigma_1&&
\\[-.2ex]
\end{array}
$$
in which $\tau_{\ell} \in
K_i^{\bb_{\ell}} I$, $\sigma_{\ell}
\in K_{i-1}^{\bb_{\ell}} I$, and
\begin{itemize}
    \item $\tau$ is boundary-linked to $\tau_0$ in the stake set $S_i^{\bb}$,
    \item $\sigma_\ell \in S_{i-1}^{\bb_\ell}$ is chain-linked to $\tau_\ell$ for all $\ell = 1,\ldots, j-1$,
    \item $\sigma_\ell = \tau_{\ell-1} - e_{k_\ell}$ for all $\ell = 1, \ldots, j-1$, where $e_{k_{\ell}} = \bb_{\ell-1} - \bb_{\ell}$, and
    \item $\sigma_j$ is cycle-linked to $\sigma$ in the shrubbery $T_{i-1}^{\aa}$.
\end{itemize}

Define $\Phi_{\sigma\tau}(\lambda)$ as the set of chain-link fences from $\tau$ to $\sigma$ along the lattice path $\lambda$.
\end{defn}

\begin{defn}\label{d:Delta}
Given a simplicial complex $K$, the following values are defined via the natural integral structure $B_{i-1}^{\mathbb{Z}}K$ on $\partial_i$.
$$%
  \Delta_i^T K = \sum_{T_i} \det(\partial_{T_i})^2,
  \quad
  \Delta_{i-1}^S K = \sum_{S_{i-1}} \det(\partial_{S_{i-1}})^2,
  \quad\text{and}\quad
  \Delta_i^\st K = (\Delta_{i-1}^S K) (\Delta_i^T K).
$$
\end{defn}

\begin{remark}\label{torsion}
Let $| \tilde{H}_{i-1} \langle T_i \rangle_{\text{tor}} |$ be the order of the torsion subgroup of the $(i-1)^{st}$ reduced homology of the subcomplex of $K$ whose facets are the faces of $T_i$.
Then $|\text{det}(\partial_{T_i})| = | \tilde{H}_{i-1} \langle T_i \rangle_{\text{tor}} |$, since both are equal to the product of the diagonal entries of the Smith normal form of $\partial_{T_i}$. 
Thus $\text{det}(\partial_{T_i})^2 = |\widetilde{H}_{i-1} \langle T_i \rangle_{\text{tor}} |^2$.
For the same reason, $\text{det}(\partial_{S_{i-1}})^2 = | \tilde{H}_{i-2} \langle S_{i-1} \rangle_{\text{tor}} |^2$.
In the three-variable case, the Koszul simplicial complexes are too small to contain any torsion, so $\text{det}(\partial_{T_i})^2 = \text{det}(\partial_{S_{i-1}})^2 = 1$.
See \cite[Theorem 3.3]{cck17} for the combinatorial formula for the Moore--Penrose pseudoinverse in \cite[Theorem 1]{berg86} written using these torsion subgroups.  
\end{remark}

\begin{remark}\label{delta formula}
Since in the three-variable case, $\text{det}(\partial_{T_i})^2 = \text{det}(\partial_{S_{i-1}})^2 = 1$, $\Delta_i^TK$ is equal to the number of shrubberies of dimension $i$ and $\Delta_{i-1}^SK$ is the number of stake sets of dimension $i-1$. 
Since each pair of shrubbery and stake set give a hedge, $\Delta_i^{ST}K$ is the number of hedges of $K$ of dimension $i$.
\end{remark}

\begin{defn}\label{d:DeltaLambda}
Given a hedgerow along $\lambda$, define the following:
$$%
\begin{array}{r@{\ }c@{\ }l@{\qquad\quad}r@{\ }c@{\ }l}
  \delta_{i,\bb} &=& \det(\partial_{S_i^\bb})
  &
  \Delta_i^{\!\bb} &=& \Delta_i^S K^\bb I,
\\\delta_{i,\bb_j} &=& \det(\partial_{S_{i-1}^{\bb_j}})\det(\partial_{T_i^{\bb_j}})
  &
  \Delta_i^{\!\bb_j} &=& \Delta_i^\st K^{\bb_j} I
  \qquad\text{ for } j = 1,\dots,\ell-1,
\\\delta_{i,\aa} &=& \det(\partial_{T_{i-1}^\aa})
  &
  \Delta_{i-1}^{\!\aa} &=& \Delta_{i-1}^T K^\aa I,
  \qquad\text{}
\\[.75ex]
  \delta_{i,\lambda} &=& \dis\prod_{j=0}^\ell \delta_{i,\bb_j}
  &
  \Delta_{i,\lambda} I &=& \dis\prod_{j=0}^\ell \Delta_i^{\!\bb_j}
\end{array}
$$
\end{defn}

\begin{remark}\label{capital Delta formula}
By Remark \ref{delta formula} and Definition \ref{d:DeltaLambda}, $\Delta_{i,\lambda}I$ is the number of hedgerows along $\lambda$.
\end{remark}

Each edge in a chain-link fence is assigned a weight, which is the coefficient of the succeeding face in the circuit, shrub, or hedge rim of the preceding face.

\begin{defn}\label{d:weight}
Let $c_\sigma(\sigma',T_{i-1})$, $c_\sigma(\tau,\st_i)$, and
$c_\rho(\rho',S_i)$ be the coefficients on~$\sigma'$, $\tau$,
and~$\rho'$ in the circuit, shrub, and hedge rim of an
$(i-1)$-face~$\sigma$, an $(i-1)$-stake~$\sigma$, and
an~$i$-face~$\rho$, respectively.
The edges in a chain-link fence are given the following \emph{weights}:
\begin{itemize}\itemsep=1ex
\item%
the boundary-link $\tau_0$\,---\,$\tau$ has weight $\delta_{i,\bb}^2
c_\tau(\tau_0,S{}_i^{\,\bb})$,
\item%
the chain-link
\raisebox{2ex}[0pt][0pt]{$\tau_j\!$}\,\,\raisebox{.5ex}{\tiny$\diagdown$}\,%
\raisebox{-1ex}[0pt][0pt]{$\sigma_j$}
has weight $\delta_{i,\bb_j}^2
c_{\sigma_j}(\tau_j,\st_i{}^{\!\!\bb_j})$,
\item%
the containment\,
\raisebox{-1.25ex}[0pt][0pt]{$\,\sigma_j$}\raisebox{.25ex}{\tiny$\diagup$}\,%
\raisebox{1.25ex}[0pt][0pt]{$\tau_{j-1}$} 
has weight $(-1)^{\sigma_j \subset \tau_{j-1}}$, and
\item%
the cycle-link $\sigma$\,---\,$\sigma_\ell$ has weight
$\delta_{i,\aa}^2\,c_{\sigma_\ell}(\sigma,T_{i-1}^\aa)$.
\end{itemize}
The \emph{weight} of the chain-link fence $\phi$ is the
product~$w_\phi$ of the weights on its edges.
\end{defn}

\begin{remark}\label{simpler weights}
By Remark \ref{torsion}, $\delta_{i,\bb}^2 = \delta_{i,\bb_j}^2 = \delta_{i,\mathbf{a}}^2 = 1$ in the three-variable case.
Therefore the edges in a chain-link fences in the three-variable case have the following simpler weights:
\begin{itemize}\itemsep=1ex
    \item the boundary-link $\tau_0$\,---\,$\tau$ has weight $c_{\tau}(\tau_0, S_i^{\bb})$,
    \item the chain-link \raisebox{2ex}[0pt][0pt]{$\tau_j\!$}\,\,\raisebox{.5ex}{\tiny$\diagdown$}\,%
\raisebox{-1ex}[0pt][0pt]{$\sigma_j$} has weight $c_{\sigma_j}(\tau_j, ST_i^{\bb_j})$,
    \item the containment \,
\raisebox{-1.25ex}[0pt][0pt]{$\,\sigma_j$}\raisebox{.25ex}{\tiny$\diagup$}\,%
\raisebox{1.25ex}[0pt][0pt]{$\tau_{j-1}$} has weight $(-1)^{\sigma_j \subset \tau_{j-1}}$, and
    \item the cycle-link $\sigma$\,---\,$\sigma_{\ell}$ has weight $c_{\sigma_{\ell}}(\sigma, T_{i-1}^{\mathbf{a}})$.
\end{itemize}
\end{remark}

\begin{thm}[\cite{emo20}, Theorem 3.7]\label{canonicalcombinatorialdescrip}
The canonical sylvan homomorphism $D^{\aa \bb}: \widetilde{C}_i K^{\bb}I \rightarrow \widetilde{C}_{i-1} K^{\aa} I$ is given by its sylvan matrix, with entries
\[ D_{\sigma \tau}^{\aa \bb} = \sum\limits_{\lambda \in \Lambda_{\aa \bb}} \frac{1}{\Delta_{i,\lambda}I} \sum\limits_{\phi \in \Phi_{\sigma \tau} (\lambda)} w_{\phi}. \]
\end{thm}

In the three-variable case, the computation of chain-link fences can be bypassed, giving an expression of the entries as sums over lattice paths of weights that depend only on the number of certain Koszul complexes at degrees indexed by each path.
These combinatorial formulas are given in Theorems~\ref{compendium0to1} and \ref{threevariabletheorem}.

\section{Sylvan matrix entries for the map $F_0 \leftarrow F_1$}\label{coefficients1to0}
\end{singlespace}

Throughout this paper, let $\kk$ be a field with $\text{char}(\kk) \neq 2,3$.
Let $x_i, x_j,$ and $x_k$ be indeterminants. 
In the three-variable case, $K^{\mathbf{b}}I$ is always a subcomplex of the simplex with facet $ijk$.
If $\beta_{1,\mathbf{b}} = \text{dim}_{\Bbbk}(\widetilde{H}_{0}(K^{\mathbf{b}}(I); \Bbbk) \neq 0$, then $K^{\mathbf{b}}(I)$ is isomorphic to either two vertices \casethree, a vertex and an edge \casetwo, or three vertices \casefour.
The main result of this section is Theorem~\ref{compendium0to1}, which describes all values $D_{\nothing v}^{\aa \bb}$.
See Theorem~\ref{threevariabletheorem} for the entries $D_{ve}^{\aa \bb}$.

\begin{thm}\label{compendium0to1}
Let $\bb$ and $\aa$ be degree vectors such that $\aa \prec \bb$, $\text{dim}_{\kk}(\widetilde{H}_0 K^{\bb}I;\kk) \neq 0$, and $\aa$ is the degree vector of a monomial generator of $I$.
Then the sylvan matrix entries $D_{\nothing v}^{\aa \bb}$ are given below.
\begin{enumerate}
    \item If $i$ is an isolated vertex in $K^{\bb}I$ and $\aa$ is the unique generator that lies behind $\bb$ in the $i$-direction, then $D_{\nothing i}^{\aa \bb} = 1$.
    \item If $\aa$ does not lie behind $\bb$ in the direction of the connected component of $v$, then $D_{\nothing v}^{\aa \bb} = 0$.
    \item Suppose $K^{\bb}I$ is \casetwo and $\aa$ lies behind $\bb$ in the direction of the edge. Given a lattice path $\lambda \in \Lambda(\aa,\bb)$, let $\bb_{\lambda}$ be the degree vector closest to $\bb$ along $\lambda$ such that $K^{\bb_{\lambda}}I$ contains one or fewer vertices. Let $n$ be the length of $\lambda$. Then
    \[D_{\nothing i}^{\aa \bb} = \frac{1}{2^n} \sum\limits_{\lambda \in \Lambda (\aa,\bb)} 2^{|\bb_{\lambda} - \aa|}. \]
\end{enumerate}
\end{thm}

\begin{proof}
These claims are Lemmas \ref{sylvan0to1}, \ref{zeromap}, and \ref{gridcount}.
\end{proof}

\begin{lemma}\label{sylvan0to1}
Let $i$ be an isolated vertex in $K^{\bb}I$, and let $\mathbf{a}$ be the unique generator that lies behind $\bb$ in the $e_i$-direction.
Then $D_{\nothing i}^{\aa \bb} = 1 $. 
\end{lemma}

\begin{proof}
Note that there is only one saturated, decreasing lattice path in the staircase from $\mathbf{b}$ to $\mathbf{c}$, namely, the lattice path that moves back in the $i$-direction at each step.
For every degree vector $\mathbf{b'}$ between $\mathbf{b}$ and $\mathbf{c}$, $K^{\mathbf{b'}}I$ is isomorphic to the single vertex $i$. 
There is only one hedgerow along this lattice path:
$$%
\begin{array}{@{}r@{\qquad\quad}c@{\,}c@{\,}c@{\,}c@{\,}c@{\,}c@{\,}c@{\,}c@{}}
   &    &\qquad\qquad&        &\qquad\qquad&     &\qquad\qquad\\[-2.5ex]
\\[-2.5ex]\lambda:&
    \mathbf{c}    &\fillbar&     \bb_{j-1}    &\cdots&   \bb_1   &\fillbar&      \bb
\\ST_0^{\,\lambda}:&
T_{-1} = \{\}&\qquad\qquad&T_0 = \{i\}&\qquad\qquad&T_0 = \{i\}&\qquad\qquad& S_0 = \{ \}
\\ &         &\qquad\qquad&S_{-1}=\{\nothing\}&\qquad\qquad&S_{-1}=\{\nothing\}
\end{array}
$$
Since the vertex $i$ does not appear with nonzero coefficient in the boundary of any edges of $K^{\bb}I$, the hedge rim $r(i)$ in $S_0 = \{\, \}$ is $i$, since $i - r(i) = i - i = 0 \in \widetilde{B}_0 (K^{\bb}I; \kk)$.
The lattice path moves back in the $e_i$-direction only.
Thus in order for the chain-link fence not to terminate, each vertex in the top line of the chain-link fence must be $i$, since $\sigma_{\ell} = \tau_{\ell-1} \setminus i$.
Note that $\nothing$ is indeed chain-linked to $i$, since the shrub $s(\nothing) = i$.
Finally, $\sigma_{\ell} = \nothing$, and since $\nothing$ is a cycle, $\zeta_{T_{-1}}(\nothing) = \nothing$, so $\nothing$ is cycle-linked to itself.
The weight of each linkage is $1$, as $c_i(i,S_0^{\bb}) = c_{\nothing}(i,ST_0^{\bb_{\ell}}) = c_{\nothing}(\nothing, T_{-1}^{\mathbf{a}}) = 1$.
Thus, along this hedgerow, there is also only one chain-link fence, beginning with $i$ and ending with $\sigma$:
$$%
\qquad\qquad
\begin{array}{*{11}{@{}c@{}}}
\\[-2.2ex]
   &         &i&           &        &         &i&           &        &         &
   i \edgehoriz 1 i
\\ &\edgeup 1& &\cdots&        && &\edgedown 1&        &\edgeup 1&
\\ \nothing\edgehoriz 1\nothing
   &         & &           &&         & &           &\nothing&         &
\\[-.2ex]
\end{array}
$$
By Remark \ref{delta formula}, $\Delta_{i,\lambda}I = 1$, and $w_{\phi} = 1$ is the product of the weights of the linkages.
Then $D_{\nothing i}^{\mathbf{c} \bb} = \sum\limits_{\lambda \in \Lambda(\mathbf{c},\bb)} \frac{1}{\Delta_{i,\lambda}I} \sum\limits_{\phi \in \Phi(\nothing, i)} w_{\phi} = 1$. \qedhere 

\end{proof}

\begin{lemma}\label{zeromap}
If $\aa$ does not lie behind $\bb$ in the direction of the connected component of a vertex $v$, then $D_{\nothing v}^{\aa \bb}=0$.
\end{lemma}

\begin{proof}
A vertex is only be boundary-linked to another vertex that lies in the same connected component.
If $\aa$ does not lie behind $\bb$ in the direction of the connected component of $v$, then no lattice path $\lambda$ from $\bb$ to $\aa$ will move back in the direction of any vertex in the connected component of $v$.
Therefore the chain-link fence will terminate.
\end{proof}

\begin{lemma}\label{gridcount}
Suppose that $K^{\mathbf{b}}(I)$ is \casetwo and $\aa$ lies behind $\bb$ in the direction of the edge $ij$.
Given a lattice path $\lambda \in \Lambda(\mathbf{a}, \bb)$, let $\bb_{\lambda}$ be the degree vector closest to $\bb$ along $\lambda$ such that the Koszul simplicial complex contains one or fewer vertices.
Let $n$ be the length of $\lambda$.
Then 
\[D_{\nothing i}^{\mathbf{a} \bb} = \frac{1}{2^n} \sum\limits_{\lambda \in \Lambda(\mathbf{a}, \bb)} 2^{|\mathbf{b}_{\lambda} - \mathbf{a}|}. \]
\end{lemma}

\begin{proof}
Consider all possible hedgerows along $\lambda \in \Lambda(\mathbf{a},\bb)$. 
In degree $\bb$, the stake set $S_0^{\bb}$ can be $\{ i \}$ or $\{ j\}$.
At each interior lattice point $\mathbf{b}_{\ell}$ such that $\bb_{\lambda} \prec \mathbf{b}_{\ell} \prec \bb$, the stake set $S_{-1}$ must be the set consisting of the empty face $\{ \nothing \}$, since $\nothing$ is a dimension $-1$ boundary.
A shrubbery $T_0^{\mathbf{b}_{\ell}}$ consists of either of the two vertices in $K^{\bb_{\ell}}I$.
Once the degree vector $\bb_{\lambda}$ is reached in the lattice path, there is only one choice of hedge for each subsequent degree vector along $\lambda$, since there is only one (or fewer) vertex in the Koszul simplicial complex.
In degree $\mathbf{a}$, the only shrubbery of dimension $-1$ is $T_{-1} = \{ \, \}$.
Therefore $\Delta_{\lambda,0}I = 2^{|\bb - \bb_{\lambda}|}$.

Next it is shown that hedgerows that yield chain-link fences are in bijection with lattice paths and with chain-link fences with initial post $i$.
Suppose in degree $\bb$, $S_0^{\bb} = \{ i \}$. 
Since the chain-link fence must start with $i$ and $i - j$ is a boundary, $r(i) = j$ and $c_i(j,S_0^{\bb}) = 1$.
Therefore the only vertex boundary-linked to $i$ is $j$.
The chain-link fence terminates unless the lattice path moves back in the $j$-direction.
If $S_0^{\bb} = \{ j \}$, then $i$ is boundary-linked to $i$ with weight $1$, since $r(i) = i$ and $c_i(i,S_0^{\bb}) = 1$.
The chain-link fence terminates unless the lattice path moves back in the $i$-direction.
In both cases the weight of the second edge in the chain-link fence is $(-1)^{\nothing \subset i} = (-1)^{\nothing \subset j} = 1$.
At each interior lattice point $\bb_{\ell}$ such that $\bb_{\lambda} \prec \bb_{\ell} \prec \bb$, the face $\nothing$ is chain-linked only to whichever vertex $v$ is in $T_0^{\bb_{\ell}}$, with weight $1$, since $s(\nothing) = v$.
At the next step, the chain-link fence terminates unless the lattice path moves back in the direction of that vertex.
Once $\bb_{\lambda}$ is reached, the vertex in the shrubbery is the same as the direction in which the lattice path moves back next.
The chain-link fence ends with $\nothing \in K^{\mathbf{a}}I$, which is cycle-linked to itself with weight $1$.
Therefore there is only one hedgerow along a given lattice path that yields a non-terminating chain-link fence, and it yields exactly one chain-link fence.
The weight $w_{\phi}=1$ for all chain-link fences $\phi$.
Thus
\begin{align*}
    D_{i}^{\mathbf{a} \bb} &= \sum\limits_{\lambda \in \Lambda(\mathbf{a},\bb)} \frac{1}{2^{|\bb - \bb_{\lambda}|}} \sum_{\phi \in \Phi_{\nothing i}(\lambda)} 1 \\
    &= \sum\limits_{\lambda \in \Lambda(\mathbf{a},\bb)} \frac{1}{2^{|\bb - \bb_{\lambda}|}}\\
    &= \sum\limits_{\lambda \in \Lambda(\mathbf{a},\bb)} \frac{2^{|\bb_{\lambda} - \mathbf{a}|}}{2^{|\bb_{\lambda} - \mathbf{a}|}2^{|\bb - \bb_{\lambda}|}}\\
    &= \frac{1}{2^n} \sum\limits_{\lambda \in \Lambda(\mathbf{a},\bb)} 2^{|\bb_{\lambda} - \mathbf{a}|}.&&\qedhere 
\end{align*}
\end{proof}

\begin{defn}\label{neighboringsyzygies}
Suppose $K^{\bb}I$ is \casetwo, and let $\{ \aa_{\ell} \}_{\ell=1}^s$ be the set of degree vectors of generators of $I$ that lie behind $\bb$ in the direction of the edge. 
Order the degree vectors $\aa_1, \aa_2, \ldots, \aa_{s}$ so that the $j^{th}$ component of $\aa_{\ell+1}$ is greater than the $j^{th}$ component of $\aa_{\ell}$ for $\ell = 1, \ldots, s-1$. 
Then the \textit{neighboring syzygies} of $\aa_{\ell}$ are the degree vectors $\text{lcm}(\aa_{\ell}, \aa_{\ell -1})$ and $\text{lcm}(\aa_{\ell}, \aa_{\ell +1})$.
\end{defn}

When $K^{\bb}I$ is \casetwo and $\mathbf{a}$ is a generator that lies behind $\bb$ in the direction of the edge $ij$, the numerator of the  sylvan matrix entries $D_{\nothing i }^{\mathbf{a} \bb}$ can be computed as a sum over certain lattice paths that each have weight $1$.
This is the content of the following corollary of Theorem \ref{compendium0to1}.

\begin{cor}\label{gridcounttheorem}
Suppose that $K^{\mathbf{b}}(I)$ is \casetwo.
Let $\{ \mathbf{a}_{\ell} \}_{\ell}$ be the set of degree vectors of generators that lie behind $\mathbf{b}$ in the $i j$-direction.
Then $D_{\nothing i}^{\mathbf{a}_{\ell} \bb}$ can be computed via the following steps:
\begin{enumerate}
\item Let $m := \text{max}_{\mathbf{a}_{\ell},\lambda \in \Lambda(\mathbf{a}_{\ell},\bb)} \{ |\bb - \bb_{\lambda} | \} $.
\item Draw an $m \times m$ grid with a diagonal line running from the bottom left corner to the top right corner. Label the generators and syzygies where they would fall on the staircase surface, with $\mathbf{b}$ in the bottom right corner.
\item $D_{\nothing i}^{\mathbf{a}_{\ell} \bb}$ is the number of saturated, decreasing lattice paths that start at $\mathbf{b}$, pass either between neighboring syzygies or through a neighboring syzygy and leaving in the direction of $\mathbf{a}_{\ell}$, and end on the diagonal divided by the total number of lattice paths from $\bb$ to the diagonal ($2^m$).
\end{enumerate}
\end{cor}

\begin{proof}
Lemma \ref{gridcount} says that the coefficients $D_{\nothing v}^{\mathbf{a} \bb}$ are sums over all saturated, decreasing lattice paths from $\bb$ to $\mathbf{a}$ of the weight associated to the lattice path, divided by $2^n$.
This weight is $2^{|\bb_{\lambda} - \mathbf{a}|}$.
This is the same as the number of lattice paths that follow along $\lambda$ until the lattice point $\bb_{\lambda}$ and then branch off in any (decreasing) direction for a length of $|\bb_{\lambda} - \mathbf{a}|$, i.e., until they reach the diagonal drawn in the $n \times n$ grid.
The lattice paths only contribute to $D_{\nothing v}^{\mathbf{a} \bb}$ if the original lattice path $\lambda$ ends at degree $\mathbf{a}$.
Therefore on the grid shown, the lattice paths must pass between neighboring syzygies of $\mathbf{a}$ or pass through the neighboring syzygies in the direction of $\mathbf{a}$.

An $m \times m$ grid can be drawn instead of an $n \times n$ grid because $2^{n - m}$ will divide every weight and can thus be canceled from the terms in the sum in \ref{gridcount}.
\end{proof}

\begin{example}\label{301}
Consider the ideal $I = \langle x^3z, xyz, y^2z, x^3y^2,x^2y^3 \rangle$ whose staircase diagram is given below.
$K^{321}I$ has facets $xy$ and $z$.
To use Corollary~\ref{gridcounttheorem}, note $m = \text{max}_{\mathbf{a},\lambda \in \Lambda(\mathbf{a},321)}\{ |321 - 321_{\lambda}| \} = 3$.
To compute the entries of $D^{301,321}$, $D^{111,321}$, and $D^{021,321}$, draw a $3 \times 3$ grid and plot the generators and neighboring syzygies.
There are five decreasing lattice paths that start at $\mathbf{321}$, pass through $\mathbf{311}$ and move up, pass between $\mathbf{311}$ and $\mathbf{121}$, or pass through $\mathbf{121}$ and move left and then end on the diagonal, as shown below.
Therefore $D_{\nothing, x}^{111,321} = D_{\nothing, y}^{111,321} = \frac{5}{2^3}$ by Corollary~\ref{gridcounttheorem}.
Similarly, $D_{\nothing x}^{301,321} = D_{\nothing y}^{301,321} = \frac{1}{4}$ and $D_{\nothing x}^{021,321} = D_{\nothing y}^{021,321} = \frac{1}{8}$.

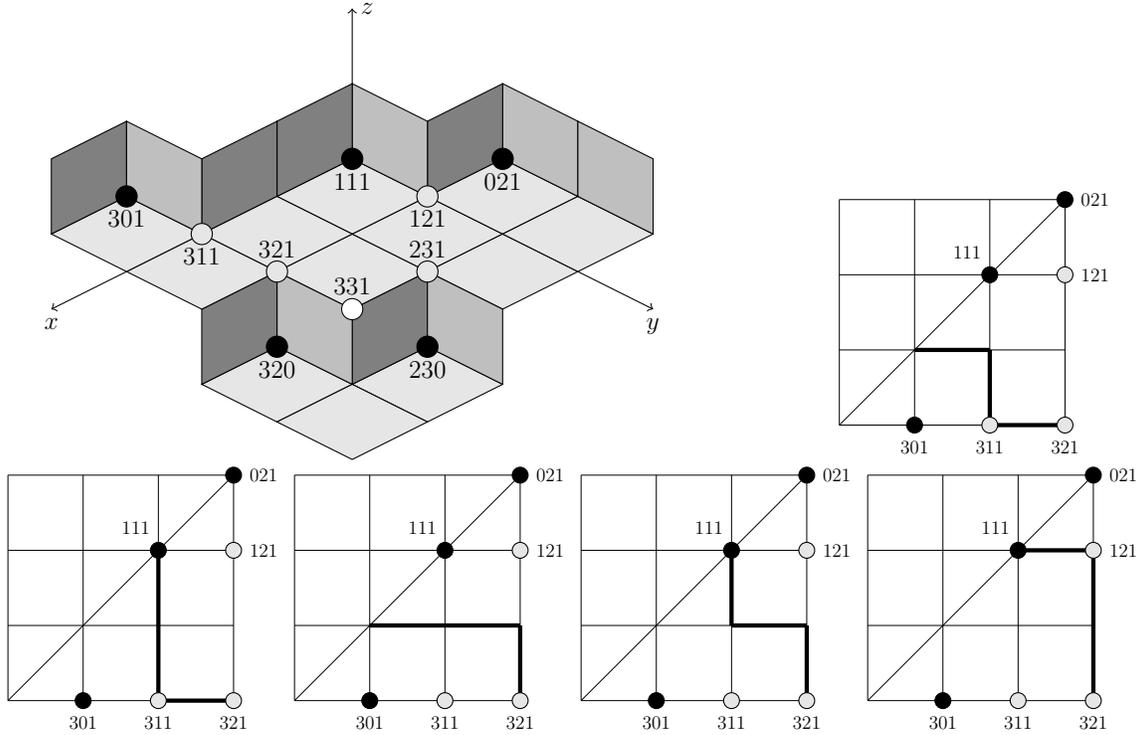
\begin{center}
\begin{tikzpicture}[scale=0.75,every node/.style={scale=0.65}]
      \draw[fill=gray!20] (0,0) -- (1,0.5) -- (0,1) -- (-1,0.5) -- (0,0);
      \draw[fill=gray!20] (2,0) -- (3,0.5) -- (2,1) -- (1,0.5) -- (2,0);
      \draw[fill=gray!20] (1,-0.5) -- (2,0) -- (1,0.5) -- (0,0) -- (1,-0.5);
      \draw[fill=gray] (1,0.5) -- (1,1.5) -- (2,2) -- (2,1) -- (1,0.5);
      \draw[fill=gray!50] (1,0.5) -- (1,1.5) -- (0,2) -- (0,1) -- (1,0.5);
      \draw[fill=gray!20] (1,1.5) -- (2,2) -- (1,2.5) -- (0,2) -- (1,1.5);
      \draw[fill=gray!20] (2,2) -- (3,2.5) -- (2,3) -- (1,2.5) -- (2,2);
      \draw[fill=gray!20] (3,2.5) -- (4,3) -- (3,3.5) -- (2,3) -- (3,2.5);
      \draw[fill=gray!20] (1,2.5) -- (2,3) -- (1,3.5) -- (0,3) -- (1,2.5);
      \draw[fill=gray!20] (3,1.5) -- (4,2) -- (3,2.5) -- (2,2) -- (3,1.5);
      \draw[fill=gray!20] (4,2) -- (5,2.5) -- (4,3) -- (3,2.5) -- (4,2);
      \draw[fill=gray!50] (3,0.5) -- (3,1.5) -- (2,2) -- (2,1) -- (3,0.5);
      \draw[fill=gray!50] (3,3.5) -- (4,3) -- (4,4) -- (3,4.5) -- (3,3.5);
      \draw[fill=gray!50] (4,3) -- (5,2.5) -- (5,3.5) -- (4,4) -- (4,3);
      \draw[fill=gray] (3,3.5) -- (3,4.5) -- (2,4) -- (2,3) -- (3,3.5);
      \draw[fill=gray!50] (2,3) -- (2,4) -- (1,4.5) -- (1,3.5) -- (2,3);
      \draw[fill=gray] (1,3.5) -- (1,4.5) -- (0,4) -- (0,3) -- (1,3.5);
      \draw[fill=gray] (0,3) -- (0,4) -- (-1,3.5) -- (-1,2.5) -- (0,3);
      \draw[fill=gray!20] (0,3) -- (-1,2.5) -- (0,2) -- (1,2.5) -- (0,3);
      \draw[fill=gray] (0,1) -- (0,2) -- (-1,1.5) -- (-1,0.5) -- (0,1);
      \draw[fill=gray!20] (0,2) -- (-1,2.5) -- (-2,2) -- (-1,1.5) -- (0,2);
      \draw[fill=gray!20] (-1,2.5) -- (-2,3) -- (-3,2.5) -- (-2,2) -- (-1,2.5);
      \draw[fill=gray] (-2,3) -- (-2,4) -- (-3,3.5) -- (-3,2.5) -- (-2,3);
      \draw[fill=gray!50] (-1,2.5) -- (-1,3.5) -- (-2,4) -- (-2,3) -- (-1,2.5);
      \filldraw (0,1) circle (4pt);
      \node at (0,0.7) {$320$};
      \filldraw (2,1) circle (4pt);
      \node at (2,0.7) {$230$};
      \filldraw (-2,3) circle (4pt);
      \node at (-2,2.7) {$301$};
      \filldraw (1,3.5) circle (4pt);
      \node at (1,3.2) {$111$};
      \filldraw (3,3.5) circle (4pt);
      \node at (3,3.2) {$021$};
      \filldraw[fill=white] (1,1.5) circle (4pt);
      \node at (1,1.8) {$331$};
      \filldraw[fill=gray!20] (0,2) circle (4pt);
      \node at (0,2.3) {$321$};
      \filldraw[fill=gray!20] (2,2) circle (4pt);
      \node at (2,2.3) {$231$};
      \filldraw[fill=gray!20] (-1,2.5) circle (4pt);
      \node at (-1,2.2) {$311$};
      \filldraw[fill=gray!20] (2,3) circle (4pt);
      \node at (2,2.7) {$121$};
      \draw[->] (1,4.5) -- (1,5.5) node[right] {$z$};
       \draw[->] (-2,2) -- (-3,1.5) node[below] {$x$};
       \draw[->] (4,2) -- (5,1.5) node[below] {$y$};
\end{tikzpicture}
\hspace{2cm}
\begin{tikzpicture}[every node/.style={scale=0.6}]
      \draw[very thin] (0,0) -- (3,0);
      \draw[very thin] (0,0) -- (0,3);
      \draw[very thin] (0,3) -- (3,3);
      \draw[very thin] (3,0) -- (3,3);
      \draw[very thin] (1,0) -- (1,3);
      \draw[very thin] (2,0) -- (2,3);
      \draw[very thin] (0,1) -- (3,1);
      \draw[very thin] (0,2) -- (3,2);
      \draw[very thin] (0,0) -- (3,3);
      
      \draw[ultra thick] (3,0) -- (2,0);
      \draw[ultra thick] (2,0) -- (2,1);
      \draw[ultra thick] (2,1) -- (1,1);
      \filldraw (1,0) circle (3pt);
      \filldraw[fill=gray!20] (2,0) circle (3pt);
      \filldraw[fill=gray!20] (3,0) circle (3pt);
      \filldraw (2,2) circle (3pt);
      \filldraw[fill=gray!20] (3,2) circle (3pt);
      \filldraw (3,3) circle (3pt);
      \node at (1,-0.3) {$301$};
      \node at (2,-0.3) {$311$};
      \node at (3,-0.3) {$321$};
      \node at (3.4,2) {$121$};
      \node at (3.4,3) {$021$};
      \node at (1.7,2.3) {$111$};

\end{tikzpicture}
\begin{tikzpicture}[every node/.style={scale=0.6}]
      \draw[very thin] (0,0) -- (3,0);
      \draw[very thin] (0,0) -- (0,3);
      \draw[very thin] (0,3) -- (3,3);
      \draw[very thin] (3,0) -- (3,3);
      \draw[very thin] (1,0) -- (1,3);
      \draw[very thin] (2,0) -- (2,3);
      \draw[very thin] (0,1) -- (3,1);
      \draw[very thin] (0,2) -- (3,2);
      \draw[very thin] (0,0) -- (3,3);
      \draw[ultra thick] (3,0) -- (2,0);
      \draw[ultra thick] (2,0) -- (2,2);
      \filldraw (1,0) circle (3pt);
      \filldraw[fill=gray!20] (2,0) circle (3pt);
      \filldraw[fill=gray!20] (3,0) circle (3pt);
      \filldraw (2,2) circle (3pt);
      \filldraw[fill=gray!20] (3,2) circle (3pt);
      \filldraw (3,3) circle (3pt);
      \node at (1,-0.3) {$301$};
      \node at (2,-0.3) {$311$};
      \node at (3,-0.3) {$321$};
      \node at (3.4,2) {$121$};
      \node at (3.4,3) {$021$};
      \node at (1.7,2.3) {$111$};

\end{tikzpicture}
\begin{tikzpicture}[every node/.style={scale=0.6}]
      \draw[very thin] (0,0) -- (3,0);
      \draw[very thin] (0,0) -- (0,3);
      \draw[very thin] (0,3) -- (3,3);
      \draw[very thin] (3,0) -- (3,3);
      \draw[very thin] (1,0) -- (1,3);
      \draw[very thin] (2,0) -- (2,3);
      \draw[very thin] (0,1) -- (3,1);
      \draw[very thin] (0,2) -- (3,2);
      \draw[very thin] (0,0) -- (3,3);
      \draw[ultra thick] (3,0) -- (3,1);
      \draw[ultra thick] (3,1) -- (1,1);
      \filldraw (1,0) circle (3pt);
      \filldraw[fill=gray!20] (2,0) circle (3pt);
      \filldraw[fill=gray!20] (3,0) circle (3pt);
      \filldraw (2,2) circle (3pt);
      \filldraw[fill=gray!20] (3,2) circle (3pt);
      \filldraw (3,3) circle (3pt);
      \node at (1,-0.3) {$301$};
      \node at (2,-0.3) {$311$};
      \node at (3,-0.3) {$321$};
      \node at (3.4,2) {$121$};
      \node at (3.4,3) {$021$};
      \node at (1.7,2.3) {$111$};

\end{tikzpicture}
\begin{tikzpicture}[every node/.style={scale=0.6}]
      \draw[very thin] (0,0) -- (3,0);
      \draw[very thin] (0,0) -- (0,3);
      \draw[very thin] (0,3) -- (3,3);
      \draw[very thin] (3,0) -- (3,3);
      \draw[very thin] (1,0) -- (1,3);
      \draw[very thin] (2,0) -- (2,3);
      \draw[very thin] (0,1) -- (3,1);
      \draw[very thin] (0,2) -- (3,2);
      \draw[very thin] (0,0) -- (3,3);
       \draw[ultra thick] (3,0) -- (3,1);
      \draw[ultra thick] (3,1) -- (2,1);
      \draw[ultra thick] (2,1) -- (2,2);
      \filldraw (1,0) circle (3pt);
      \filldraw[fill=gray!20] (2,0) circle (3pt);
      \filldraw[fill=gray!20] (3,0) circle (3pt);
      \filldraw (2,2) circle (3pt);
      \filldraw[fill=gray!20] (3,2) circle (3pt);
      \filldraw (3,3) circle (3pt);
      \node at (1,-0.3) {$301$};
      \node at (2,-0.3) {$311$};
      \node at (3,-0.3) {$321$};
      \node at (3.4,2) {$121$};
      \node at (3.4,3) {$021$};
      \node at (1.7,2.3) {$111$};

\end{tikzpicture}
\begin{tikzpicture}[every node/.style={scale=0.6}]
      \draw[very thin] (0,0) -- (3,0);
      \draw[very thin] (0,0) -- (0,3);
      \draw[very thin] (0,3) -- (3,3);
      \draw[very thin] (3,0) -- (3,3);
      \draw[very thin] (1,0) -- (1,3);
      \draw[very thin] (2,0) -- (2,3);
      \draw[very thin] (0,1) -- (3,1);
      \draw[very thin] (0,2) -- (3,2);
      \draw[very thin] (0,0) -- (3,3);
      \draw[ultra thick] (3,0) -- (3,2);
      \draw[ultra thick] (3,2) -- (2,2);
      \filldraw (1,0) circle (3pt);
      \filldraw[fill=gray!20] (2,0) circle (3pt);
      \filldraw[fill=gray!20] (3,0) circle (3pt);
      \filldraw (2,2) circle (3pt);
      \filldraw[fill=gray!20] (3,2) circle (3pt);
      \filldraw (3,3) circle (3pt);
      \node at (1,-0.3) {$301$};
      \node at (2,-0.3) {$311$};
      \node at (3,-0.3) {$321$};
      \node at (3.4,2) {$121$};
      \node at (3.4,3) {$021$};
      \node at (1.7,2.3) {$111$};

\end{tikzpicture}
\end{center}
\end{example}

\begin{singlespace}
\section{Coefficients in the maps from 
$F_2$ to $F_1$}\label{coefficientsfrom2to1}
\end{singlespace}
The coefficients in the maps $\widetilde{H}_0(K^{\aa}I;\kk) \leftarrow \widetilde{H}_1 (K^{\bb}I;\kk)$ are determined by examining all sequences of Koszul complexes that can occur along a lattice path from $\bb$ to $\aa$.
These sequences are given in Figure \ref{cases}.
Let $D_{ve}^{\aa \bb, \lambda}$ be the contribution to $D_{ve}^{\aa \bb}$ along the lattice path $\lambda$, where $v$ is a vertex and $e$ is an edge.
The coefficients of the sylvan matrix are obtained by adding the contributions along all lattice paths from $\bb$ to $\aa$.
The theorem below describes all sylvan matrix entries $D_{ve}^{\aa \bb}$ in the three-variable case, where $v$ is a vertex and $e$ is an edge.
Define $(-1)^{i,\prod_{i_l} i_l}:= \text{sign}(i, \prod_{i_l} i_l)$, the sign the face of $i_1 \cdots \hat{i} \cdots i_l$ in $\partial(i_1 \cdots i_l)$.

\begin{thm}\label{threevariabletheorem}
Let $\bb$ and $\aa$ be degree vectors such that $\aa \preceq \bb$, $K^{\bb}I$ is \casezero, and $\text{dim}_{\kk} \widetilde{H}_0 (K^{\aa}I;\kk) \neq 0$.
Let $\lambda = (\aa = \bb_j, \ldots, \bb_1, \bb)$ be a saturated, decreasing lattice path such that $\bb - \bb_1 = e_i$.
Let $r$ be the number of times the lattice path $\lambda$ passes through a degree vector $\mathbf{c}$ such that $K^{\mathbf{c}}I$ is \caseone, let $s$ be the number of times $\lambda$ passes through a degree vector $\mathbf{c}$ such that $K^{\mathbf{c}}I$ is \casefive, and let $m$ be the number of times $\lambda$ passes through a degree vector $\mathbf{c}$ such that $K^{\mathbf{c}}I$ is \casetwo.
Let
\begin{itemize}
    \item $s_{i,ij} := (-1)^{j,ij}$,
    \item $s_{j,ij} := (-1)^{i,ij}$, and
    \item $s_{k,ij} := (-1)^{j,ij}$.
\end{itemize}
Then
\begin{enumerate}
    \item $D_{v,jk}^{\mathbf{a} \mathbf{b}, \lambda} = 0$ for all $v$.
    \item If $K^{\mathbf{a}}I$ is \casethreeprime, then $D_{v,ij}^{\mathbf{a} \mathbf{b}} = s_{v,ij} \frac{1}{2}$ for any vertex $v$ of $K^{\mathbf{a}}I$.
    \item If $K^{\mathbf{a}}I$ is \casetwo or \casefour, then $D_{i,ij}^{\mathbf{a}\mathbf{b}} = s_{i,ij} \frac{1}{3^{r+1}}$, $D_{j,ij}^{\aa \bb} = s_{j,ij} \frac{3^{r+1}+1}{2 \cdot 3^{r+1}}$, and $D_{k,ij}^{\aa \bb} = s_{k,ij} \frac{3^{r+1}-1}{2 \cdot 3^{r+1}}$. Here $j$ and $k$ can also be swapped.
    \item If $K^{\mathbf{a}}I$ is \casethree, then $D_{v,ij}^{\mathbf{a}\mathbf{b}, \lambda} = s_{v,ij} \frac{3^r+1}{3^r 2^{m+s+2}}$ for any vertex $v$ in $K^{\mathbf{a}}I$, $D_{i,ik}^{\mathbf{a} \mathbf{b},\lambda} = (-1)^{i,ik} \frac{3^r-1}{3^r 2^{m+s+2}}$, and $D_{j,ik}^{\mathbf{a} \mathbf{b},\lambda} = (-1)^{k,ik} \frac{3^r-1}{3^r 2^{m+s+2}}$.
\end{enumerate}
\end{thm}

\begin{proof}[Proof of \ref{threevariabletheorem}]
Statement $1$ follows from Lemma~\ref{zeroentrylemma}.
Statement $2$ follows from Lemmas~\ref{zerotothreeprime} and \ref{zerotoonetothreeprime}.
Statement $3$ follows from Lemmas~\ref{zerotofourortwo} and \ref{zerotoonetofour}.
Statement $4$ follows from Lemmas~\ref{zerototwotothree}, \ref{zerototwotofivetothree}, \ref{zerotoonetothree}, \ref{zerotoonetotwotothree}, and \ref{zerotoonetotwotofivetothree}.
\end{proof}

\subsection{Lattice path cases}\label{latticecases}
Suppose $K^{\mathbf{b}}I$ is \casezero (see the staircase diagram below).

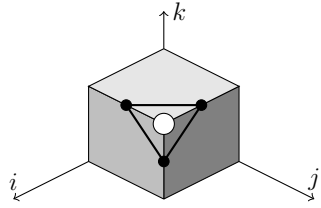
\begin{center}
\begin{tikzpicture}[scale=1,every node/.style={scale=0.8}]
      \draw[fill=gray] (1,0.5) -- (1,1.5) -- (2,2) -- (2,1) -- (1,0.5);
      \draw[fill=gray!50] (1,0.5) -- (1,1.5) -- (0,2) -- (0,1) -- (1,0.5);
      \draw[fill=gray!20] (1,1.5) -- (2,2) -- (1,2.5) -- (0,2) -- (1,1.5);
      \draw[->] (0,1) -- (-1,0.5) node[above] {$i$};
      \draw[->] (2,1) -- (3,0.5) node[above] {$j$};
      \draw[->] (1,2.5) -- (1,3) node[right] {$k$};
      
      \draw[thick] (0.5,1.75) -- (1.5,1.75) -- (1,1) -- (0.5,1.75);
      \filldraw (0.5,1.75) circle (2pt);
      \filldraw (1.5,1.75) circle (2pt);
      \filldraw (1,1) circle (2pt);
      \filldraw[fill=white] (1,1.5) circle (4pt);
\end{tikzpicture}
\end{center}

Let $\mathbf{b_{\ell}}$ such that $\mathbf{a} \preceq \bb_{\ell} \preceq \bb$ and $\text{dim}_{\Bbbk} (\widetilde{H}_0 (K^{\mathbf{a}}I; \Bbbk))$ is nonzero.
Suppose further that $\mathbf{b}_{\ell}$ is a degree vector in a lattice path $(\aa = \bb_j, \bb_{j-1}, \ldots, \bb_1, \bb)$ where $\bb - \bb_1 = i$.
(By symmetry, the results also hold for lattice paths that initially move back in the $j$- or $k$-direction, exchanging each $i$ and $j$ or $k$, respectively.) 
Then there are several cases for $K^{\mathbf{b}_{\ell}}I$.
These cases are given below, with the degree $\bb_{\ell}$ marked by a white dot and the Koszul simplicial complexes drawn on the staircase diagram.
\begin{itemize}
    \item \caseone the edges $ij$ and $ik$
    
    
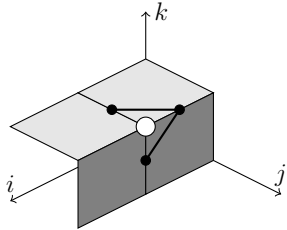
\begin{tikzpicture}[scale=0.9,every node/.style={scale=0.8}]
     \draw[fill=gray!20] (1,1.5) -- (0,2) -- (-1,1.5) -- (0,1) -- (1,1.5);
     \draw[fill=gray] (1,1.5) -- (1,0.5) -- (0,0) -- (0,1) -- (1,1.5);
      \draw[fill=gray!20] (1,1.5) -- (2,2) -- (1,2.5) -- (0,2) -- (1,1.5);
      \draw[fill=gray] (1,1.5) -- (2,2) -- (2,1) -- (1,0.5) -- (1,1.5);
      
      \draw[thick] (0.5,1.75) -- (1.5,1.75) -- (1,1);
      \filldraw (0.5,1.75) circle (2pt);
      \filldraw (1.5,1.75) circle (2pt);
      \filldraw (1,1) circle (2pt);
   
      \filldraw[fill=white] (1,1.5) circle (4pt);
      \draw[->] (0,0.9) -- (-1,0.4) node[above] {$i$};
      \draw[->] (2,1) -- (3,0.5) node[above] {$j$};
      \draw[->] (1,2.5) -- (1,3.2) node[right] {$k$};
\end{tikzpicture}
   
    \item \casetwo the edge $ij$ and the vertex $k$
    
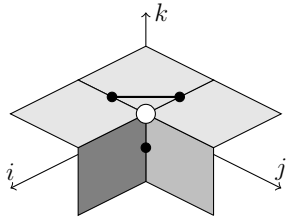
    \begin{tikzpicture}[scale=0.9,every node/.style={scale=0.8}]
      \draw[fill=gray!50] (1,0.5) -- (1,1.5) -- (0,2) -- (0,1) -- (1,0.5);
      \draw[fill=gray!20] (1,1.5) -- (2,2) -- (1,2.5) -- (0,2) -- (1,1.5);
      \draw[fill=gray!20] (0,3) -- (-1,2.5) -- (0,2) -- (1,2.5) -- (0,3);
      \draw[fill=gray] (0,1) -- (0,2) -- (-1,1.5) -- (-1,0.5) -- (0,1);
      \draw[fill=gray!20] (0,2) -- (-1,2.5) -- (-2,2) -- (-1,1.5) -- (0,2);
      
      \draw[thick] (-0.5,2.25) -- (0.5,2.25);
      \filldraw (-0.5,2.25) circle (2pt);
      \filldraw (0.5,2.25) circle (2pt);
      \filldraw (0,1.5) circle (2pt);

      \filldraw[fill=white] (0,2) circle (4pt);
      \draw[->] (-1,1.4) -- (-2,0.9) node[above] {$i$};
      \draw[->] (1,1.4) -- (2,0.9) node[above] {$j$};
      \draw[->] (0,3) -- (0,3.5) node[right] {$k$};
\end{tikzpicture}
    \item \casethree the vertices $i$ and $j$
    
    
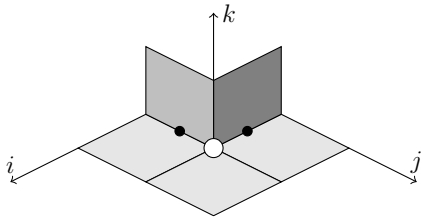
\begin{tikzpicture}[scale=0.9,every node/.style={scale=0.8}]
      \draw[fill=gray!20] (0,0) -- (1,0.5) -- (0,1) -- (-1,0.5) -- (0,0);
      \draw[fill=gray!20] (2,0) -- (3,0.5) -- (2,1) -- (1,0.5) -- (2,0);
      \draw[fill=gray!20] (1,-0.5) -- (2,0) -- (1,0.5) -- (0,0) -- (1,-0.5);
      \draw[fill=gray] (1,0.5) -- (1,1.5) -- (2,2) -- (2,1) -- (1,0.5);
      \draw[fill=gray!50] (1,0.5) -- (1,1.5) -- (0,2) -- (0,1) -- (1,0.5);
      
      \filldraw (0.5,0.75) circle (2pt);
      \filldraw (1.5,0.75) circle (2pt);
      
      \filldraw[fill=white] (1,0.5) circle (4pt);
      \draw[->] (-1,0.5) -- (-2,0) node[above] {$i$};
      \draw[->] (3,0.5) -- (4,0) node[above] {$j$};
      \draw[->] (1,1.5) -- (1,2.5) node[right] {$k$};
\end{tikzpicture}
    \item \casethreeprime the vertices $j$ and $k$
    
    
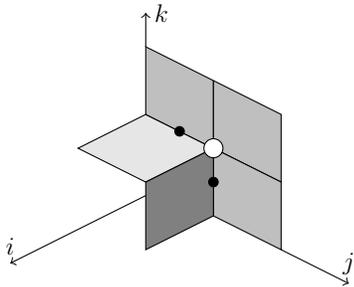
\begin{tikzpicture}[scale=0.9,every node/.style={scale=0.8}]
      \draw[fill=gray!20] (1,2.5) -- (2,3) -- (1,3.5) -- (0,3) -- (1,2.5);
      \draw[fill=gray!50] (1,3.5) -- (1,4.5) -- (2,4) -- (2,3) -- (1,3.5);
      \draw[fill=gray!50] (2,4) -- (2,3) -- (3,2.5) -- (3,3.5) -- (2,4);
      \draw[fill=gray!50] (2,3) -- (2,2) -- (3,1.5) -- (3,2.5) -- (2,3);
      \draw[fill=gray] (2,3) -- (1,2.5) -- (1,1.5) -- (2,2) -- (2,3);
      
      \filldraw (1.5,3.25) circle (2pt);
      \filldraw (2,2.5) circle (2pt);
    
      \filldraw[fill=white] (2,3) circle (4pt);
      \draw[->] (1,2.3) -- (-1,1.3) node[above] {$i$};
      \draw[->] (3,1.5) -- (4,1) node[above] {$j$};
      \draw[->] (1,4.5) -- (1,5) node[right] {$k$};
\end{tikzpicture}
\newpage 
    \item \casefour the vertices $i$, $j$, and $k$
    
    
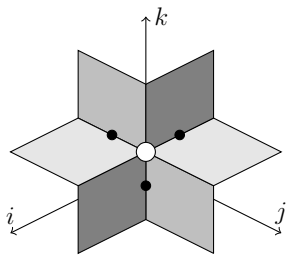
\begin{tikzpicture}[scale=0.9,every node/.style={scale=0.8}]
      \draw[fill=gray!50] (1,0.5) -- (1,1.5) -- (0,2) -- (0,1) -- (1,0.5);
      \draw[fill=gray!20] (1,1.5) -- (2,2) -- (1,2.5) -- (0,2) -- (1,1.5);
      \draw[fill=gray!50] (0,3) -- (0,2) -- (-1,2.5) -- (-1,3.5) -- (0,3);
      \draw[fill=gray] (0,3) -- (0,2) -- (1,2.5) -- (1,3.5) -- (0,3);
      \draw[fill=gray] (0,1) -- (0,2) -- (-1,1.5) -- (-1,0.5) -- (0,1);
      \draw[fill=gray!20] (0,2) -- (-1,2.5) -- (-2,2) -- (-1,1.5) -- (0,2);
      
      \filldraw (-0.5,2.25) circle (2pt);
      \filldraw (0.5,2.25) circle (2pt);
      \filldraw (0,1.5) circle (2pt);

      \filldraw[fill=white] (0,2) circle (4pt);
      \draw[->] (-1,1.3) -- (-2,0.8) node[above] {$i$};
      \draw[->] (1,1.3) -- (2,0.8) node[above] {$j$};
      \draw[->] (0,3) -- (0,4) node[right] {$k$};
\end{tikzpicture}
    \item \casefive the edge $ij$
    
     
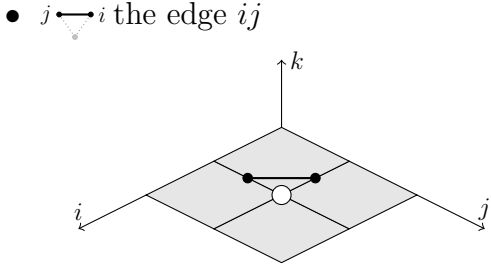
\begin{tikzpicture}[scale=0.9,every node/.style={scale=0.8}]
      \draw[fill=gray!20] (1,1.5) -- (2,2) -- (1,2.5) -- (0,2) -- (1,1.5);
      \draw[fill=gray!20] (2,2) -- (3,2.5) -- (2,3) -- (1,2.5) -- (2,2);
      \draw[fill=gray!20] (1,2.5) -- (2,3) -- (1,3.5) -- (0,3) -- (1,2.5);
      \draw[fill=gray!20] (0,3) -- (-1,2.5) -- (0,2) -- (1,2.5) -- (0,3);
      
      \draw[thick] (0.5,2.75) -- (1.5,2.75);
      \filldraw (0.5,2.75) circle (2pt);
      \filldraw (1.5,2.75) circle (2pt);

      \filldraw[fill=white] (1,2.5) circle (4pt);
      \draw[->] (-1,2.5) -- (-2,2) node[above] {$i$};
      \draw[->] (3,2.5) -- (4,2) node[above] {$j$};
      \draw[->] (1,3.5) -- (1,4.5) node[right] {$k$};
\end{tikzpicture}
\end{itemize}

The lemmas below are used to determine all possible sequences of Koszul complexes along lattice paths beginning at \casezero and ending at a degree vector $\aa$ such that $\widetilde{H}_0 (K^{\aa}I; \kk) \neq 0$.
These sequences are given in Figure~\ref{cases}.

\begin{lemma}\label{Fthm}
If $\mathbf{b}$ and $\mathbf{a}$ lie on the staircase and $\mathbf{a} \preceq \mathbf{b}$, then $F:= \text{supp}(\mathbf{b} - \mathbf{a}) \in K^{\mathbf{b}}I$.
\end{lemma}

\begin{proof}
$K^{\mathbf{b}}I = \{ \text{supp}(\mathbf{b} - \mathbf{c} ) \mid \mathbf{c} \text{ is on the staircase and } \mathbf{c} \preceq \mathbf{b} \}$.
\end{proof}

\begin{lemma}\label{facet}
If $F$ is a facet of $K^{\mathbf{b}}I$, then $K^{\mathbf{a}}I \subseteq F$.
\end{lemma}

\begin{proof}
Let $\text{supp}(\mathbf{b} - \mathbf{c}) \in K^{\mathbf{a}}I$. 
Since $F = \text{supp}(\mathbf{b} - \mathbf{a})$ is a facet of $K^{\mathbf{b}}I$, $\text{supp}(\mathbf{b} - \mathbf{c}) \preceq \text{supp}(\mathbf{b} - \mathbf{a})$ for all $\mathbf{c} \preceq \mathbf{b}$ on the staircase.
Since $\text{supp}(\mathbf{c} - \mathbf{a}) \preceq \text{supp} (\mathbf{b} - \mathbf{a})$, $\text{supp}(\mathbf{c} - \mathbf{a})$ is a subface of $F$.
\end{proof}

\begin{lemma}\label{star}
$K^{\mathbf{a}}I \subseteq \text{star}(F, K^{\mathbf{b}}I)$, the set of faces of $K^{\mathbf{b}}I$ containing $F$.
\end{lemma}

\begin{proof}
Moving back from $\mathbf{a}$ by $G$ is the same as moving back from $\mathbf{b}$ by $F \cup G$.
\end{proof}

\begin{prop}
The diagram in Figure \ref{cases} shows all cases for lattice paths starting at the Koszul simplicial complex \casezero and ending at a degree vector $\aa$ such that $\widetilde{H}_0 (K^{\aa}I;\kk) \neq 0$.
\end{prop}

\begin{proof}
Start at a degree vector $\bb$ such that $K^{\bb}I$ is  \casezero and move back in the $i$-direction to degree $\mathbf{b}_1$.
Then $F = \text{supp}(\mathbf{b} - \mathbf{b}_1) = i$. 
By Lemma~\ref{star}, $K^{\mathbf{b}_1}I \subseteq \text{star}(F, K^{\mathbf{b}}I)$, so $K^{\mathbf{b}_1}I$ must be a subcomplex of \caseone. 
Since $ik$ is in $K^{\mathbf{b}}I$, the degree vectors $\mathbf{b} - i - k$ and $\mathbf{b} - i$ are on the staircase, which means that $\text{supp}((\mathbf{b} - i) - (\mathbf{b} - i - k)) = k \in K^{\mathbf{b - i}}I = K^{\mathbf{b}_1}I$ by Lemma~\ref{Fthm}.
Therefore \casezero cannot lead to \casethree or \casefive.

The complex \caseone can lead to any of \caseone, \casetwo, \casethree, \casethreeprime, \casefour, or \casefive.
Moving back along a single edge (in the $i$-, $j$-, or $k$-direction) from $\mathbf{b}$ to $\mathbf{b}_1$, $\text{supp}(\bb - \bb_1)= i$, $j$, or $k$.
Since $K^{\mathbf{b}_1}I \subseteq \text{star}(\text{supp}(\bb - \bb_1, K^{\mathbf{b}}I)$, \caseone can lead to \caseone in the $i$-direction only.
Similarly, any time $j$ and $k$ are both in $K^{\mathbf{b}_1}I$, the lattice path can only have moved back in the $i$-direction.
This is the case when $K^{\bb_1}I$ is \caseone, \casetwo, \casethreeprime, or \casefour.
Since $K^{\mathbf{b - e_i}}I$ must contain $j$ and $k$, the lattice path can only get to \casethree or \casefive by moving back in the $j$-direction.

The relevant subcomplexes of \casetwo are \casetwo, \casethree, \casethreeprime, \casefour, or \casefive.
If the lattice path moves back in the $i$- or $j$-direction, the support of the difference of the two neighboring degree vectors is $i$ or $j$. Therefore \casetwo cannot lead to \casethreeprime, since $K^{\mathbf{b}_{\ell}}I$ cannot contain $k$ by Lemma~\ref{star}.
Similarly, if the lattice path moves back in the $k$-direction, $K^{\mathbf{b}_{\ell}}I$ cannot contain $i$ or $j$.
Therefore \casetwo also cannot be followed by \casetwo or \casefour.
However, \casetwo can lead to \casethree or \casefive.

Subcomplexes of \casefive are either \casethree or \casefive, and both of these sequences are possible along a lattice path.

By Lemma~\ref{star}, \casethree, \casethreeprime, and \casefour cannot lead to any case where the reduced $0^{th}$ homology is nonzero.
\end{proof}

\fig
\begin{tikzpicture}
      \node[scale=1.5] at (-0.9,0) {\casezero};
      \draw[->] (0,0) -- (2,9);
      \node[scale=1.5] at (2.9,9) {\caseone};
      \node[scale=2] at (2.9,9.9) {$\circlearrowright$};
      \node at (2.9,9.9) {$i$};
      \draw[->] (0,0) -- (2,4);
      \node[right, scale=1.5] at (2,4) {\casetwo};
      \draw[->] (0,0) -- (2,1);
      \node[right,scale=1.5] at (2,1.2) {\casethreeprime};
      \draw[->] (0,0) -- (2,-1);
      \node[right, scale=1.5] at (2,-1) {\casefour};
      \draw[->] (3.8,9) -- (5.5,13);
      \node[right, scale=1.5] at (5.5,13) {\casetwo};
      \draw[->] (3.8,9) -- (5.5,11);
      \node[right, scale=1.5] at (5.5,11) {\casethree};
      \draw[->] (3.8,9) -- (5.5,9.5);
      \node[right, scale=1.5] at (5.5,9.5) {\casethreeprime};
      \draw[->] (3.8,9) -- (5.5,8.1);
      \node[right, scale=1.5] at (5.5,8.1) {\casefour};
      \draw[->] (3.8,9) -- (5.5,7);
      \node[right, scale=1.5] at (5.5,7) {\casefive};
      \draw[->] (7.3,7) -- (9,7);
      \node[right, scale=1.5] at (9,7) {\casethree};
      \node[scale=2] at (6.5,6.15) {$\circlearrowright$};
      \node at (6.5,5.7) {$i,j$};
      \node at (8,7.2) {$i,j$};
      \draw[->] (3.8,4) -- (5.5,4.5);
      \node[right, scale=1.5] at (5.5,4.5) {\casethree};
      \draw[->] (3.8,4) -- (5.5,3.5);
      \node[right, scale=1.5] at (5.5,3.5) {\casefive};
      \node[scale=2] at (6.5,2.6) {$\circlearrowright$};
      \node at (6.5,2.1) {$i,j$};
      \draw[->] (7.3,13) -- (9,13);
      \node[right,scale=1.5] at (9,15) {\casethree};
      \node at (8,13.99) {$i,j$};
      \draw[->] (7.3,13) -- (9,14.6);
      \node[right, scale=1.5] at (9,13) {\casefive};
      \node at (8,13.2) {$i,j$};
      \node[scale=2] at (10,12) {$\circlearrowright$};
      \node at (10,11.5) {$i,j$};
      \draw[->] (7.3,3.5) -- (9,3.5);
      \node[right, scale=1.5] at (9,3.5) {\casethree};
      \node at (8,3.7) {$i,j$};
      \draw[->] (10.8,13) -- (11.8,13);
      \node[right, scale=1.5] at (11.8,13) {\casethree};
      \node at (11.2,13.2) {$i,j$};
      \node at (1,5.1) {$i$};
      \node at (1,2.5) {$i$};
      \node at (1,0.8) {$i$};
      \node at (1,-0.3) {$i$};
      \node at (4.6,11.5) {$i$};
      \node at (4.6,10.4) {$j$};
      \node at (4.6,9.5) {$i$};
      \node at (4.6,8.8) {$i$};
      \node at (4.6,8.3) {$j$};
      \node at (4.65,4.5) {$i,j$};
      \node at (4.65,3.9) {$i,j$};
      
    \end{tikzpicture}
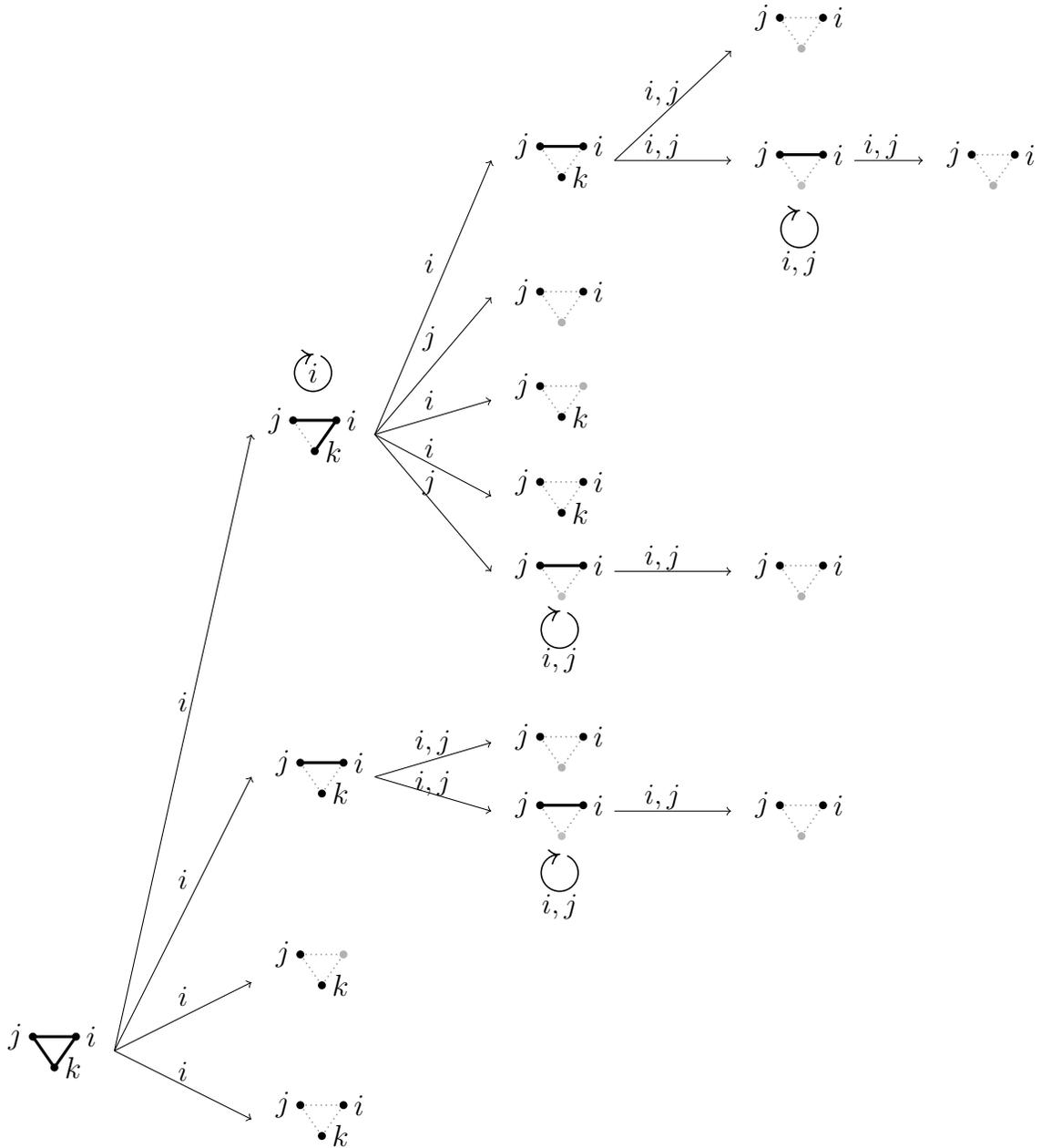
    \caption{All cases for Koszul complexes along lattice paths that begin by moving back in the $i$-direction contributing to the maps $F_1 \leftarrow F_2$}\label{cases}
    \efig

\subsection{Sylvan matrix entries $D_{ve}^{\aa \bb}$ when $\aa = \bb - ne_i$}

When $K^{\bb}I$ is \casezero, $K^{\aa}I$ is \casetwo, \casethreeprime, or \casefour, and $\aa \preceq \bb$, there is exactly one saturated, decreasing lattice path from $\bb$ to $\aa$.
This is because the lattice path must move back in the $i$-direction at each step (see Figure \ref{cases}).
The lemmas in this section give the entries $D_{ve}^{\aa \bb}$ of the sylvan matrices, where $e$ is an edge and $v$ is a vertex.
These lemmas contribute to the proof of Theorem~\ref{threevariabletheorem}.

As before, let $(-1)^{i,\prod_{i_l} i_l}:= \text{sign}(i, \prod_{i_l} i_l)$, the sign the face of $i_1 \cdots \hat{i} \cdots i_l$ in $\partial(i_1 \cdots i_l)$.

\begin{lemma}\label{zerotofourortwo}
Suppose $\bb$ and $\aa := \bb-e_i$ are degree vectors such that $K^{\bb}I$ is  \casezero and $K^{\aa}I$ is  \casetwo or \casefour.
Then the sylvan matrix is given below.
$$%
\HH_0 K^{\aa} \!\otimes\! \langle \xx^{\aa} \rangle
\xleftarrow{
\monomialmatrix
  {i\\ j\\ k}
  {\begin{array}{@{}l@{\ }c@{\ }r@{}}
    \quad\, i j &\ \ i k & j k\ \,
    \\
         \frac{(-1)^{j,ij}}{3}&\,\,  \frac{(-1)^{k,ik}}{3}  &  0
      \\ \frac{2(-1)^{i,ij}}{3}&\,\, \frac{(-1)^{k,ik}}{3}& 0
      \\ \frac{(-1)^{j,ij}}{3}&\,\, \frac{2(-1)^{i,ik}}{3} &  0
   \end{array}}
  {\\\\\\}
\!\!\!}
{\HH_1 K^{\bb} \!\otimes\! \langle \xx^{\bb} \rangle}
$$
\end{lemma}

\begin{proof}
The hedgrow along $\lambda$ consists of a stake set $S_1^{\bb}$ and a shrubbery $T_0^{\aa}$.
The stake set $S_1^{\bb}$ is forced to be the empty set $\{ \, \}$, since the only $1$-boundary is $0$.
The shrubbery $T_0^{\aa}$ can contain any one of the vertices $i$, $j$, or $k$.
Therefore the number of hedgerows is $3$, so $\Delta_{1,\lambda}I = 3$.
Given the stake set $S_1^{\bb}$, since the only $1$-boundary is $0$, the hedge rim is $r(ij) = ij$, and $c_{ij}(ij, S_1^{\bb}) = 1$.
Given the shrubbery $T_0^{\aa} = \{ j \}$, the circuit $\zeta_{T_0^{\aa}} (j) = 0$, so the chain-link fence terminates.
If $T_0^{\aa} = \{i \}$ or $\{ k\}$, then $\zeta_{T_0^{\aa}} (j) = j-i$ or $j-k$, respectively.
In these cases, $c_{j}(j, T_0^{\aa}) = 1$ and $c_{j}(v, T_0^{\aa}) = -1$ for the vertex $v$ in the shrubbery.
All chain-link fences with initial post $ij$ are given in the diagram below. 
\begin{singlespace}
$$
\begin{array}{*{4}{@{}c@{}}}
\\[-2.2ex]
    &     &     & i j \edgehoriz 1 i j
\\  &   &\edgeup{(-1)^{i,ij}}&
\\
\begin{array}{@{}r@{\ }c@{\ \,}c@{\,}c@{}}
\\[-6.3ex]
  & j\!\!\!
\\[-1.3ex]
   T_0 = \{i\}\colon
  & & \diagdown\hspace{.1ex}\raisebox{1ex}{\tiny\mkl{\hspace{-2ex}\scriptstyle 1}}
\\[-1.7ex]
  & i \raisebox{.75ex}{\tiny\mkl{\,\scriptstyle -1}} & \fillbar
\\[-1.6ex]
  &                                                  &          & j
\\[-1.6ex]
  & j \raisebox{-.5ex}{\tiny\mkl{\ \ \scriptstyle 1}} & \fillbar
\\[-1.7ex]
   T_0 = \{k\}\colon
  & & \diagup\hspace{.1ex}\raisebox{-.5ex}{\tiny\mkl{\hspace{-2.5ex}\scriptstyle -1}}
\\[-1.5ex]
  & k \!\!\!
\\[-3.6ex]
\end{array}
    &&&\\
  &&&\mko{S_1 = \{ \, \}}
\\[-.2ex]
\end{array}
$$
\end{singlespace}
The lattice paths with initial post $i k$ are the same once every $j$ and $k$ are swapped.
Any chain-link fence with initial post $j k$ will immediately terminate, since $\lambda$ moves back in the $i$-direction.
\end{proof}

\begin{lemma}\label{zerotothreeprime}
Suppose $\bb$ and $\aa := \bb - e_i$ are degree vectors such that $K^{\bb}I$ is \casezero and $K^{\aa}I$ is \casethreeprime.
Then the sylvan matrix is given below.
$$%
\HH_0 K^{\aa} \!\otimes\! \langle \xx^{\aa} \rangle
\xleftarrow{
\monomialmatrix
  {j\\ k}
  {\begin{array}{@{}l@{\ }c@{\ }r@{}}
    \quad\, i j &\ \ i k & j k\ \,
    \\
       \frac{(-1)^{i,ij}}{2} & \,\, \frac{(-1)^{k,ik}}{2} & 0
       \\ \frac{(-1)^{j,ij}}{2} & \,\, \frac{(-1)^{i,ik}}{2} & 0 
   \end{array}}
  {\\\\}
\!\!\!}
{\HH_1 K^{\bb} \!\otimes\! \langle \xx^{\bb} \rangle}
$$
\end{lemma}

\begin{proof}
The stake set $S_1^{\bb}$ is forced to be the empty set $\{ \, \}$, and $T_0^{\aa}$ can consist of either vertex $j$ or $k$.
The first three faces and corresponding edge weights are the same as in the proof of Lemma~\ref{zerotofourortwo}, since $K^{\bb}I$ is the same and $\lambda$ moves back in the $i$-direction.
All chain-link fences with initial post $i j$ are shown below. 
The terminal posts and corresponding weights follow from the arguments given in the proof of Lemma~\ref{zerotofourortwo} when $T_0^{\aa} = \{ k \}$.
\begin{singlespace}
$$
\begin{array}{*{4}{@{}c@{}}}
\\[-2.2ex]
    &     &     & i j \edgehoriz 1 i j
\\  &   &\edgeup{(-1)^{i,ij}}&
\\
\begin{array}{@{}r@{\ }c@{\ \,}c@{\,}c@{}}
\\[-6.3ex]
  & j\!\!\!
\\[-1.3ex]
  & & \diagdown\hspace{.1ex}\raisebox{1ex}{\tiny\mkl{\hspace{-2ex}\scriptstyle 1}}
\\[-1.6ex]
  &                                                  &          & j
\\[-1.7ex]
  & & \diagup\hspace{.1ex}\raisebox{-.5ex}{\tiny\mkl{\hspace{-2.5ex}\scriptstyle -1}}
\\[-1.5ex]
  & k \!\!\!
\\[-3.6ex]
\end{array}
    &&&\\
    &&&\\
  T_0 = \{ k \} &&&\mko{S_1 = \{ \}}
\\[-.2ex]
\end{array}
$$
\end{singlespace}
The chain-link fences with initial post $ik$ are the same after swapping each $j$ and $k$.
Since the lattice path moves back in the $i$-direction, there are no chain-link fences with initial post $jk$.
\end{proof}

\begin{defn}\label{partial chain-link fence}
Let $\lambda = (\bb_j, \bb_{j-1}, \ldots, \bb_1, \bb)$ be a lattice path.
Given a stake set $S_i^{\bb} \subseteq K^{\bb}I$ and a hedge $ST_i^{\mathbf{c}} \subseteq K^{\mathbf{c}}I$ for each degree $\mathbf{c} = \bb_1, \ldots, \bb_j$, a \textit{partial chain-link fence} is a sequence of faces 
$$%
\begin{array}{*{15}{@{}c@{}}}
\\[-3.2ex]
&&\!\!\tau_{j}\!\!\!&&&&\ \cdots\ &&&&\tau_1&&&&\tau_0\hbox{ --- }\tau\\
&&&\bs\,\,\,&&/&&\bs&&/&&\bs&&/&\\
&&&&\!\!\!\sigma_{j}\!\!&&&&\sigma_2&&&&\sigma_1&&
\\[-.2ex]
\end{array}
$$
in which $\tau_{\ell} \in
K_i^{\bb_{\ell}} I$, $\sigma_{\ell}
\in K_{i-1}^{\bb_{\ell}} I$, and
\begin{itemize}
    \item $\tau$ is boundary-linked to $\tau_0$ in the stake set $S_i^{\bb}$,
    \item $\sigma_\ell \in S_{i-1}^{\bb_\ell}$ is chain-linked to $\tau_\ell$ for all $\ell = 1,\ldots, j$, and
    \item $\sigma_\ell = \tau_{\ell-1} - e_{k_\ell}$ for all $\ell = 1, \ldots, j-1$, where $e_{k_{\ell}} = \bb_{\ell-1} - \bb_{\ell}$.
    \end{itemize}
The partial weight $w_{\phi}$ for a partial chain-link fence is the product of the weights on its edges, as given in Definition~\ref{d:weight}.    
\end{defn}

The definitions for a partial chain-link fence and a chain-link fence agree until degree $\bb_{j}$.
Every chain-link fence contains a partial chain-link fence.

\begin{lemma}\label{zerotoone}
If $\bb$ and $\bb_r := \bb - r \cdot e_i$ are degree vectors such that $K^{\bb}I$ is \casezero, and in the lattice path $\lambda = (\bb_{r}, \bb_{r-1}, \ldots, \bb_1, \bb)$, the Koszul simplicial complexes in degrees $\bb_1, \ldots, \bb_r$ are \caseone, there are $3^r$ partial chain-link fences with initial post $ij$.
The number of these chain-link fences that end with $i j$ is $\lceil \frac{3^r}{2} \rceil$, with partial weight $\lceil \frac{3^r}{2} \rceil$, and the number that end with $ik$ is $\lfloor \frac{3^r}{2} \rfloor$, with partial weight $(-1)^{i,ijk+1} \lfloor \frac{3^r}{2} \rfloor$. 
\end{lemma}

\begin{proof}
The stake set $S_1^{\bb}$ must be the empty set.
For $\bb_{\ell}$, $\ell = 1, \ldots, r$,  the shrubbery $T_1^{\bb_{\ell}} = \{ ij, ik \}$, and $S_0^{\bb_{\ell}}$ can be any pair of vertices, giving three options for each $\bb_{\ell}$.

In degree $\bb_{\ell}$, the stake $j$ is chain-linked to $i j$ if $S_0^{\bb_{\ell}} = \{ j, k \}$ with weight $(-1)^{i,ij}$ since $s(j) = ij$.
If $S_0^{\bb_{\ell}} = \{ j, i \}$, then $j$ is chain-linked to $ij$ and $ik$, since the shrub $s(j)$ is a linear combination of $ij$ and $ik$ whose boundary has coefficient $1$ on $j$ and coefficient $0$ on $i$.
If $S_0^{\bb_{\ell}} = \{ i, k \}$, then $j$ is not a stake and thus is not chain-linked to any face.
In this case the chain-link fence terminates.
Since the Koszul complex is symmetric, the stake $k$ is chain-linked to $i k$ if $S_0^{\bb_{\ell}} = \{ k, j \}$ and to $i k$ and $i j$ if $S_0^{\bb_{\ell}} = \{ k, i \}$.

Therefore, for each chain-link fence where the face $\sigma_{\ell}$ is equal to $j$, there are two chain-link fences where $\tau_{\ell}$ is equal to $ij$ and one where $\tau_{\ell}$ is $ik$.
Similarly, for each chain-link fence with $\sigma_{\ell} = k$, there are two chain-link fences with $\tau_{\ell} = ik$ and one with $\tau_{\ell} = ij$.
Each chain-link fence with $\tau_{\ell} = ij$ gives one chain-link fence, with $\sigma_{\ell+1} = j$, since $\lambda$ moves back in the $i$-direction.
If $\tau_{\ell} = ik$, this gives one chain-link fence with $\sigma_{\ell+1} = k$.

The proof is by induction.
For the base case, note that after starting at initial post $ij$ which is boundary-linked to itself, and moving back in the $i$-direction, there are $2 = \lceil \frac{3^1}{2} \rceil$ chain-link fences with  $\sigma_1 = j$ and $\tau_1 = i j$ as explained above.
There is only $1 = \lfloor \frac{3^1}{2} \rfloor$ chain-link fence with $\sigma_1 = j$ and $\tau_1 = ik$.

Let $A_n$ be the number of chain-link fences with $\tau_n = ij$, and let $B_n$ be the number of chain-link fences with at $\tau_n = ik$.
Assume $A_n = \lceil \frac{3^n}{2} \rceil$ and $B_n = \lfloor \frac{3^n}{2} \rfloor$.
After moving back again in the $i$-direction, there will be $A_n$ chain-link fences with $\sigma_{n+1} = j$ and $B_n$ with $\sigma_{n+1} = k$.
As described above, this produces $2(A_n) + B_n$ chain-link fences with $\tau_{n+1} = ij$ and $A_n + 2(B_n)$ with $\tau_{n+1} = i k$.
Since $3^n$ is odd, $\lceil \frac{3^n}{2} \rceil = \frac{3^n + 1}{2}$ and $\lfloor \frac{3^n}{2} \rfloor = \frac{3^n-1}{2}$.
This gives
\[ A_{n+1} = 2\cdot A_n + B_n
    = 2 \Bigl\lceil \frac{3^n}{2} \Bigr\rceil + \Bigl\lfloor \frac{3^n}{2} \Bigr\rfloor
    = \frac{3^{n+1}+1}{2}
    = \Bigl\lceil \frac{3^{n+1}}{2} \Bigr\rceil.\]
Since the total number of chain-link fences is $3^n$, $B_{n+1} = 3^n - \lceil \frac{3^{n+1}}{2} \rceil = \lfloor \frac{3^{n+1}}{2} \rfloor$.

Next, it is shown that $c_j (ik, \{j,i\}) = (-1)^{i,ij} (-1)^{i,ijk+1}$.
Note that
\begin{align*}
    \partial ((-1)^{i,ij} ((-1)^{i,ijk+1} ik + ij)) &= (-1)^{i,ij} (-1)^{i,ijk+1} ((-1)^{i,ik} k + (-1)^{k,ik} i) \\
    &\,\,\,\, + (-1)^{i,ij} (-1)^{i,ij} j + (-1)^{i,ij} (-1)^{j,ij} i\\
    &= (-1)^{i,ij}(-1)^{i,ijk+1} ((-1)^{i,ik} k + (-1)^{k,ik} i)\\
    &\,\,\,\,  + j - i. 
\end{align*}
If in the order of the variables, $i > j,k$ or $i < j,k$, then $(-1)^{i,ij}(-1)^{k,ik} = -1$ and $(-1)^{i,ijk+1} = -1$.
If $j>i>k$ or $k>i>j$, then $(-1)^{i,ij}(-1)^{k,ik} = 1$ and $(-1)^{i,ijk+1} = 1$.
Therefore $(-1)^{i,ij} (-1)^{i,ijk+1} (-1)^{k,ik} = 1$.
Then
\begin{align*}
    \partial ((-1)^{i,ij} ((-1)^{i,ijk+1} i k + i j)) &= (-1)^{i,ij} (-1)^{i,ijk+1} (-1)^{i,ik} k + j.
\end{align*}
Thus $c_j(ik, \{ j,i\}) =(-1)^{i,ij} (-1)^{i,ijk+1}$.
By symmetry, $c_k (ij, \{ k,i\}) = (-1)^{i,ik} (-1)^{i,ijk+1}$. 

There are four cases for segments $\tau_{\ell}$ --- $\sigma_{\ell}$ --- $\tau_{\ell-1}$ of the chain-link fence:
\begin{singlespace}
$$%
\begin{array}{@{}c@{}c@{}c@{}c@{}c@{}}
i j &&&&i j\\
\,\,{\scriptstyle (-1)^{i,ij}}& \edgedown{} & & \edgeup{(-1)^{i,ij}} & \\
&&j&&
\end{array}
\hspace{1cm}
\begin{array}{@{}c@{}c@{}c@{}c@{}c@{}}
i k &&&&i k\\
\,\,{\scriptstyle (-1)^{i,ik}}& \edgedown{} & & \edgeup{(-1)^{i,ik}} & \\
&&k&&
\end{array}
\hspace{1cm}
\begin{array}{@{}c@{}c@{}c@{}c@{}c@{}}
i k &&&&i j\\
& \edgedown{a} & & \edgeup{(-1)^{i,ij}} & \\
&&j&&
\end{array}
\hspace{1cm}
\begin{array}{@{}c@{}c@{}c@{}c@{}c@{}}
i j &&&&i k\\
& \edgedown{b} & & \edgeup{(-1)^{i,ik}} & \\
&&k&&
\end{array}
$$
\end{singlespace}
where $a = (-1)^{i,ij}(-1)^{i,ijk +1}$ and $b = (-1)^{i,ik}(-1)^{i,ijk+1}$.
The signs in the first two cases cancel.
In the third and fourth cases, the weights cancel to $(-1)^{i,ijk+1}$.

Therefore, for a chain-link fence with initial post $i j$, each time $\tau_{\ell-1}$ and $\tau_{\ell}$ do not agree, the total weight is multiplied by $(-1)^{i,ijk+1}$.
If this occurs an odd number of times, $\tau_r = ik$ and the partial weight is $w_{\phi} = (-1)^{i,ijk+1}$.
If this occurs an even number of times, $\tau_r = ij$ and the partial weight is $w_{\phi} = 1$.
Therefore the sum of the weights of all chain-link fences with $\tau_r = ij$ is $\lceil \frac{3^r}{2} \rceil$, and the sum of the weights of all chain-link fences with $\tau_r = ik$ is $(-1)^{i,ijk+1} \lfloor \frac{3^r}{2} \rfloor$.
\end{proof}

\begin{lemma}\label{zerotoonetothreeprime}
If $\bb$ and $\aa$ are degree vectors such that $K^{\bb}I$ is  \casezero and $K^{\aa}I$ is \casethreeprime, and in the lattice path $\lambda = (\aa, \bb_{r}, \bb_{r-1}, \ldots, \bb_1, \bb)$, all degree vectors $\bb_1, \ldots, \bb_r$ have Koszul complexes \caseone, then the contribution of $\lambda$ to the sylvan matrix is given below.
$$%
\widetilde{H}_0 K^{\aa} \!\otimes\! \langle \xx^{\aa} \rangle
\xleftarrow{
\monomialmatrix
  {j\\ k}
  {\begin{array}{@{}l@{\ }c@{\ }r@{}}
    \quad\, i j &\ \ i k &\, \, j k\ \,
    \\
       \frac{(-1)^{i,ij}}{2}& \frac{(-1)^{k,ik}}{2} & 0
       \\\frac{(-1)^{j,ij}}{2}& \frac{(-1)^{i,ik}}{2}  & 0 
   \end{array}}
  {\\\\}
\!\!\!}
{\widetilde{H}_1 K^{\bb} \!\otimes\! \langle \xx^{\bb} \rangle}
$$
\end{lemma}

\begin{proof}
In the set of partial chain-link fences along $\lambda$ from $\aa$ to $\bb_r$ with initial post $ij$, there are $\lceil \frac{3^r}{2} \rceil$ with $\tau_r = ij$ by Lemma \ref{zerotoone}.
After moving back in the $i$-direction, these are linked to $j$ with weight $(-1)^{i,ij}$.
There are $\lfloor \frac{3^r}{2} \rfloor$ partial chain-link fences with $\tau_r = ik$, and these are linked to $k$ with weight $(-1)^{i,ik}$.
The shrubbery $T_0^{\aa}$ can be $\{ j \}$ or $\{ k\}$, meaning $\Delta_{1,\lambda}I = 2 \cdot 3^r$.
If $T_0^{\aa} = \{ k \}$, then $\zeta_{T_0^{\aa}}(j) = j - k$, so $c_j (j, T_0^{\aa}) = 1$ and $c_j (k, T_0^{\aa}) = -1$. 
In this case, $\zeta_{T_0^{\aa}}(k) = 0$.
Similarly, if $T_0^{\aa} = \{ j \}$, then $\zeta_{T_0^{\aa}}(k) = k - j$, so $c_k(k, T_0^{\aa}) = 1$ and $c_k (j, T_0^{\aa}) = -1$. 
Also $\zeta_{T_0^{\aa}} (j) = 0$.
Thus by Lemma~\ref{zerotoone} there are $\lceil \frac{3^r}{2} \rceil$ chain-link fences along $\lambda$ with terminal post $j$ with total weight $w_{\phi} = (-1)^{i,ij}$ and $\lfloor \frac{3^r}{2} \rfloor$ with total weight $w_{\phi} = -(-1)^{i,ik} (-1)^{i,ijk+1}$.

This gives 
\[D_{j, ij}^{\aa \bb} = \frac{1}{2 \cdot 3^r} ( (-1)^{i,ij} \Bigl\lceil \frac{3^r}{2} \Bigr\rceil - (-1)^{i,ik}(-1)^{i,ijk+1} \Bigl\lfloor \frac{3^r}{2} \Bigr\rfloor).\]
Note that $(-1)^{i,ij}(-1)^{i,ik}(-1)^{i,ijk+1} = -1$ (see proof of Lemma \ref{zerotoone}).
Then
\begin{align*}
    \frac{1}{2 \cdot 3^r} \Big( (-1)^{i,ij} \Bigl\lceil \frac{3^r}{2} \Bigr\rceil - (-1)^{i,ik}(-1)^{i,ijk+1} \Bigl\lfloor \frac{3^r}{2} \Bigr\rfloor \Big) &= \frac{(-1)^{i,ij}}{2 \cdot 3^r} \Big( \Bigl\lceil \frac{3^r}{2} \Bigr\rceil + \Bigl\lfloor \frac{3^r}{2} \Bigr\rfloor \Big)\\
    &= \frac{(-1)^{i,ij} 3^r}{2 \cdot 3^r}\\
    &= \frac{(-1)^{i,ij}}{2}.
\end{align*}
Thus $D_{j, ij}^{\aa \bb} = \frac{(-1)^{i,ij}}{2}$.

Similarly, there are $\lceil \frac{3^r}{2} \rceil$ chain-link fences with terminal post $k$ with weight $ w_{\phi} = -(-1)^{i,ij}$ and $\lfloor \frac{3^r}{2} \rfloor$ with weight $w_{\phi} = (-1)^{i,ik}(-1)^{i,ijk+1}$.
This gives $D_{k, ij}^{\aa \bb} = \frac{(-1)^{j,ij}}{2}$.

The other entries follow by symmetry.
\end{proof}

\begin{lemma}\label{zerotoonetofour}
If $\bb$ and $\aa$ are degree vectors such that $K^{\bb}I$ is \casezero and $K^{\aa}I$ is \casetwo or \casefour, and in the lattice path $\lambda = (\aa = \bb_{r+1}, \bb_{r}, \bb_{r-1}, \ldots, \bb_1, \bb)$, $K^{\bb_{\ell}}I$ is \caseone for $\ell = i, \ldots, r$, then the sylvan matrix entries are given below.
$$%
\HH_0 K^{\aa} \!\otimes\! \langle \xx^{\aa} \rangle
\xleftarrow{
\monomialmatrix
  {i\\ j\\k}
  {\begin{array}{@{}l@{\ }c@{\ }r@{}}
    \quad\, i j &\ \ i k &\, \, j k\ \,
    \\
       \frac{(-1)^{j,ij}}{3^{r+1}} & \frac{(-1)^{k,ik}}{3^{r+1}} &0\\
       (-1)^{i,ij} \frac{3^r + \lceil \frac{3^r}{2} \rceil}{3^{r+1}} & (-1)^{k,ik} \frac{3^r + \lfloor \frac{3^r}{2} \rfloor}{3^{r+1}}&0\\
       (-1)^{j,ij} \frac{3^r + \lfloor \frac{3^r}{2} \rfloor}{3^{r+1}} & (-1)^{i,ik} \frac{3^r + \lceil \frac{3^r}{2} \rceil}{3^{r+1}}&0
   \end{array}}
  {\\\\\\}
\!\!\!}
{\HH_1 K^{\bb} \!\otimes\! \langle \xx^{\bb} \rangle}
$$
\end{lemma}

\begin{proof}
To get from \caseone to \casetwo or \casefour, the chain-link fence must move back in the $i$-direction.
A partial chain-link fence with $\tau_r = ij$ will have $\sigma_{r+1} = j$.
There are three options for $T_0^{\aa}$: $\{ i \}$, $\{j\}$, and $\{k \}$.
If $T_0^{\aa} = \{ i \}$, then $\zeta_{T_0^{\aa}}(j) = j-i$, so $c_j(j, T_0^{\aa}) = 1$ and $c_j (i, T_0^{\aa}) = -1$.
If $T_0^{\aa} = \{ k \}$, then $\zeta_{T_0^{\aa}}(j) = j - k$, so $c_j(j, T_0^{\aa}) = 1$ and $c_j (k, T_0^{\aa}) = -1$.
If $T_0^{\aa} = \{ j \}$, then $\zeta_{T_0^{\aa}}(j) = 0$, so $j$ is not cycle-linked to any vertex.
By symmetry, these links and weights apply to chain-link fences with $\tau_r = ik$, with every $j$ and $k$ swapped.

This gives, in total, $\lceil \frac{3^r}{2} \rceil$ chain-link fences with terminal post $i$ and weight $ -(-1)^{i,ij}$ and $\lfloor \frac{3^r}{2} \rfloor$ with weight $ -(-1)^{i,ik} (-1)^{i,ijk+1} = (-1)^{i,ij}$, giving a total weight of $w_{\phi} = -(-1)^{i,ij}(\lceil \frac{3^r}{2} \rceil - \lfloor \frac{3^r}{2} \rfloor) = 1(-1)^{i,ij} = (-1)^{j,ij}$.
Similarly, there are $2 \cdot \lceil \frac{3^r}{2} \rceil$ with terminal post $j$ with weight $(-1)^{i,ij}$ and $\lfloor \frac{3^r}{2} \rfloor$ with weight $-(-1)^{i,ik}(-1)^{i,ijk+1} = (-1)^{i,ij}$, giving a total weight of $w_{\phi} = (-1)^{i,ij}( 2 \cdot \lceil \frac{3^r}{2} \rceil +  \lfloor \frac{3^r}{2} \rfloor) = (-1)^{i,ij} (3^r + \lceil \frac{3^r}{2} \rceil)$.
Finally, there are $\lceil \frac{3^r}{2} \rceil$ with terminal post $k$ with weight $-(-1)^{i,ij}$ and $2 \cdot \lfloor \frac{3^r}{2} \rfloor$ with weight $(-1)^{i,ik}(-1)^{i,ijk+1} = (-1)^{j,ij}$, giving a total weight of $w_{\phi} = -(-1)^{i,ij}(\lceil \frac{3^r}{2} \rceil + 2 \cdot \lfloor \frac{3^r}{2} \rfloor) = (-1)^{j,ij}(3^r + \lfloor \frac{3^r}{2} \rfloor)$.

There are $3$ options for $T_0^{\aa}$, meaning $\Delta_{1,\lambda}I = 3^{r+1}$.
\end{proof}

\subsection{Sylvan matrix entries $D_{ve}^{\aa \bb, \lambda}$ along a single, non-unique lattice path}

Often there are multiple lattice paths between degree vectors $\bb$ and $\aa$.
The lemmas in this section give the contributions $D_{ve}^{\aa \bb, \lambda}$ along the lattice path $\lambda$, and the sylvan matrix entry $D_{ve}^{\aa \bb}$ is the sum over all lattice paths $\lambda$ from $\bb$ to $\aa$ of the contributions $D_{ve}^{\aa \bb, \lambda}$.

\begin{lemma}\label{zerototwotothree}
Suppose $\bb$, $\bb_1 := \bb-e_i$, and $\aa := \bb_1 - e_i$ or $\aa:= \bb_1 - e_j$ are degree vectors such that $K^{\bb}I$ is  \casezero, $K^{\bb_1}I$ is \casetwo, and $K^{\aa}I$ is \casethree.
Then the contribution to the sylvan matrix from $\lambda = (\aa, \bb_1, \bb)$ is given below.
$$%
\HH_0 K^{\aa} \!\otimes\! \langle \xx^{\aa} \rangle
\xleftarrow{
\monomialmatrix
  {i\\ j}
  {\begin{array}{@{}l@{\ }c@{\ }r@{\ }}
    \quad\, i j &\ \ i k &\, \, j k\ \,
    \\
       \frac{(-1)^{j,ij}}{4} & \,\, 0 & 0
       \\ \frac{(-1)^{i,ij}}{4} & \,\, 0 & 0 
   \end{array}}
  {\\\\}
\!\!\!}
{\HH_1 K^{\bb} \!\otimes\! \langle \xx^{\bb} \rangle}
$$
\end{lemma}

\begin{proof}
Since the only $1$-boundary in $K^{\bb}I$ is $0$, $S_1^{\bb}$ must be the empty set.
Since there is only one edge in $K^{\bb_1}I$, the shrubbery $T_1^{\bb_1}$ must be $\{ ij \}$, and $S_0^{\bb_1}$ can be either $\{j \}$ or $\{i\}$.
The shrubbery $T_0^{\aa}$ can be $\{ i\}$ or $\{j\}$.
Therefore $\Delta_{1,\lambda}I = 4$.
All chain-link fences for the case $\aa = \bb_1 - e_i$ with initial post $i j$ are given below.
\begin{singlespace}
$$%
\begin{array}{*{7}{@{}c@{}}}
\\[-2.2ex]
    &         & ij&           &   &         & ij \edgehoriz 1 ij
\\  &\edgeup{} &  &\edgedown{} &   &\edgeup{} &
\\
\begin{array}{@{}r@{\ }c@{\ \,}c@{\,}c@{}}
\\[-6.3ex]
  & j\!\!\!
\\[-1.3ex]
  & & \diagdown\hspace{.1ex}\raisebox{1ex}{\tiny\mkl{\hspace{-2ex}\scriptstyle 1}}
\\[-1.6ex]
  &                                                  &          & j
\\[-1.7ex]
  & & \diagup\hspace{.1ex}\raisebox{-.5ex}{\tiny\mkl{\hspace{-2.5ex}\scriptstyle -1}}
\\[-1.5ex]
  & i\!\!\!
\\[-3.6ex]
\end{array}
  &           &  &             & j &         &
\\[1ex]
\\[1ex]
  T_0 = \{ i \} & &&&&\mko{S_0 = \{j\}}
\\[-.2ex]
\end{array}
$$
\end{singlespace}
The unlabeled edges have weight $(-1)^{i,ij}$.

All chain-link fences for $\aa = \bb_1 - e_j$ with initial post $ij$ are given below:
\begin{singlespace}
$$%
\begin{array}{*{7}{@{}c@{}}}
\\[-2.2ex]
    &         &ij&           &   &         &ij \edgehoriz 1 ij
\\  &\edgeup{c}\,\,\,\, &  &\edgedown{} &   &\edgeup{} &
\\
\begin{array}{@{}r@{\ }c@{\ \,}c@{\,}c@{}}
\\[-6.3ex]
  & i\!\!\!
\\[-1.3ex]
  & & \diagdown\hspace{.1ex}\raisebox{1ex}{\tiny\mkl{\hspace{-2ex}\scriptstyle 1}}
\\[-1.6ex]
  &                                                  &          & i
\\[-1.7ex]
  & & \diagup\hspace{.1ex}\raisebox{-.5ex}{\tiny\mkl{\hspace{-2.5ex}\scriptstyle -1}}
\\[-1.5ex]
  & j\!\!\!
\\[-3.6ex]
\end{array}
  &           &  &             & j &         &
\\[1ex]
\\[1ex]
  T_0 = \{ j \} & &&&&\mko{S_0 = \{j\}}
\\[-.2ex]
\end{array}
$$
\end{singlespace}
where $c = (-1)^{j,ij}$, and the unlabeled edges have weight $(-1)^{i,ij}$.

Note that in either case, a chain-link fence with initial post $ik$ will terminate, since after moving back in the $i$-direction, $\sigma_1 = k$, which is not a vertex in any stake set $S_0^{\bb_1}$.
\end{proof}

\begin{lemma}\label{zerototwotofivetothree}
Suppose $\bb$ and $\aa$ are degree vectors such that $K^{\bb}I$ is \casezero, $K^{\aa}I$ is \casethree, and a lattice path $\lambda$ from $\bb$ to $\aa$ passes through \casetwo and then \casefive $r$ times.
Then the contribution to the sylvan matrix along the lattice path is given below.
$$%
\HH_0 K^{\aa} \!\otimes\! \langle \xx^{\aa} \rangle
\xleftarrow{
\monomialmatrix
  {i\\ j}
  {\begin{array}{@{}l@{\ }c@{\ }r@{}}
    \quad\, i j &\ \ i k &\, \, j k\ \,
    \\
       \frac{(-1)^{j,ij}}{2^{r+2}}  & 0 & 0
       \\\frac{(-1)^{i,ij}}{2^{r+2}}& 0 & 0 
   \end{array}}
  {\\\\}
\!\!\!}
{\HH_1 K^{\bb} \!\otimes\! \langle \xx^{\bb} \rangle}
$$
\end{lemma}

\begin{proof}
The choices of $ST_1$ for \casetwo and \casefive are the same, as $T_1$ must be $\{ ij \}$, and $S_0$ must consist of a single vertex with nonzero coefficient in the boundary of the edge $ij$, either $\{i\}$ or $\{ j\}$.
In \casethree, $T_0$ can be either $\{i \}$ or $\{ j \}$.
Thus $\Delta_{1,\lambda} = 2^{r+2}$, where $r$ is the number of lattice points $\bb_{\ell}$ such that $K^{\bb_{\ell}}I$ is \casefive.

Starting with initial post $ij$, $\tau_{\ell}$ can only be the edge $ij$ for any $\ell$, since it is the only edge that appears in $K^{\bb_{\ell}}I$ for $\bb_{\ell} \neq \bb$.
The faces $\sigma_{\ell}$ for $\bb_{\ell} \neq \mathbf{a}$ can be $i$ or $j$, depending on which direction the lattice path moves back in.
The weight of the link between $ij$ and the vertex $v$ is equal to $c_ij (v, ST_1^{\bb_{\ell}})$, as both are the coefficient of $v$ in $\partial(ij)$.
Therefore the signs corresponding to these two weights will cancel.
If $\bb_{r+1} - \aa = e_i$, then the weight of the link between $ij$ and $j$ is $(-1)^{i,ij}$.
If $T_0^{\aa} = \{ i\}$, then $\zeta_{T_0^{\aa}}(j) = j-i$.
See the chain-link fences below.
If $T_0^{\aa} = \{ j \}$, then $\zeta_{T_0^{\aa}}(j) = 0$, so the chain-link fence terminates. 
Thus $w_{\phi} = (-1)^{i,ij}$ when the terminal post is $j$ and $w_{\phi} = -(-1)^{i,ij} = (-1)^{j,ij}$ when the terminal post is $i$.
When $\bb_{r+1} - \aa = e_j$, swap the $i$ and $j$.
\begin{singlespace}
$$%
\begin{array}{*{7}{@{}c@{}}}
\\[-2.2ex]
    &         &ij&           &   &         &ij \edgehoriz 1 ij
\\  &\edgeup{} &  &\edgedown{} &   &\edgeup{} &
\\
\begin{array}{@{}r@{\ }c@{\ \,}c@{\,}c@{}}
\\[-6.3ex]
  & j\!\!\!
\\[-1.3ex]
  & & \diagdown\hspace{.1ex}\raisebox{1ex}{\tiny\mkl{\hspace{-2ex}\scriptstyle 1}}
\\[-1.6ex]
  &                                                  &          & j
\\[-1.7ex]
  & & \diagup\hspace{.1ex}\raisebox{-.5ex}{\tiny\mkl{\hspace{-2.5ex}\scriptstyle -1}}
\\[-1.5ex]
  & i\!\!\!
\\[-3.6ex]
\end{array}
  &           &  &             & j &         &
\\[1ex]
\\[1ex]
  T_0 = \{ i \} & &&&&\mko{S_0 = \{j\}}
\\[-.2ex]
\end{array}
$$
\end{singlespace}
\end{proof}

\begin{lemma}\label{zerotoonetothree}
If $\bb$ and $\aa$ are degree vectors such that $K^{\bb}I$ is \casezero and $K^{\aa}I$ is \casethree, and in the lattice path $\lambda = (\aa, \bb_{r}, \bb_{r-1}, \ldots, \bb_1, \bb)$, all degree vectors $\bb_1, \ldots, \bb_r$ have Koszul complexes \caseone, then the contribution of $\lambda$ to the sylvan matrix is given below.
$$%
\HH_0 K^{\aa} \!\otimes\! \langle \xx^{\aa} \rangle
\xleftarrow{
\monomialmatrix
  {i\\ j}
  {\begin{array}{@{}l@{\ }c@{\ }r@{}}
    \quad\, ij &\ \ ik &\, \, jk\ \,
    \\
       \frac{(-1)^{j,ij}}{2 \cdot 3^r} \lceil \frac{3^r}{2} \rceil & \frac{(-1)^{i,ik}}{2 \cdot 3^r} \lfloor \frac{3^r}{2} \rfloor  & 0
       \\\frac{(-1)^{i,ij}}{2 \cdot 3^r} \lceil \frac{3^r}{2} \rceil & \frac{(-1)^{k,ik}}{2 \cdot 3^r} \lfloor \frac{3^r}{2} \rfloor  & 0 
   \end{array}}
  {\\\\}
\!\!\!}
{\HH_1 K^{\bb} \!\otimes\! \langle \xx^{\bb} \rangle}
$$
\end{lemma}

\begin{proof}
The lattice path must have $\bb_r - \aa = e_j$, so only the partial chain-link fences as in Lemma~\ref{zerotoone} with $\tau_r =ij$ will contribute to the entries of the sylvan matrix.
After moving back in the $i$-direction, there are now $\lceil \frac{3^r}{2} \rceil$ chain-link fences with $\sigma_{r+1} = i$ with partial weight $(-1)^{j,ij}$.
The shrubbery $T_0^{\aa}$ can be $\{ i \}$ or $\{ j \}$.
If $T_0^{\aa} = \{ j \}$, then $i$ is cycle-linked to $i$ with weight $1$ and to $j$ with weight $-1$.
Thus $D_{i, ij}^{\aa \bb, \lambda} = \frac{(-1)^{j,ij}}{2 \cdot 3^r} \lceil \frac{3^r}{2} \rceil$ and $D_{j, ij}^{\aa \bb, \lambda} = -\frac{(-1)^{j,ij}}{2 \cdot 3^r} \lceil \frac{3^r}{2} \rceil$.

If the chain-link fence has initial post $ik$, then by Lemma~\ref{zerotoone} there are $\lfloor \frac{3^r}{2} \rfloor$ partial chain-link fences with $\tau_r = ij$.
By Lemma~\ref{zerotoone} and the argument above, $D_{i, ik}^{\aa \bb, \lambda} = \frac{(-1)^{j,ij} (-1)^{i,ijk+1}}{2 \cdot 3^r} \lfloor \frac{3^r}{2} \rfloor = \frac{(-1)^{i,ik}}{2 \cdot 3^r} \lfloor \frac{3^r}{2} \rfloor$ and $D_{j, ik}^{\aa \bb, \lambda} = -\frac{(-1)^{j,ij} (-1)^{i,ijk+1}}{2 \cdot 3^r} \lfloor \frac{3^r}{2} \rfloor = \frac{(-1)^{k,ik}}{2 \cdot 3^r} \lfloor \frac{3^r}{2} \rfloor$.
\end{proof}

\begin{lemma}\label{zerotoonetofivetothree}
Suppose $\bb$ and $\aa$ are degree vectors such that $K^{\bb}I$ is \casezero and $K^{\aa}I$ is \casethree. Suppose also that $\lambda = (\aa, \mathbf{c}_s, \ldots, \mathbf{c}_1, \bb_r, \ldots, \bb_1, \bb)$ is a lattice path such that $K^{\mathbf{c}_{\ell}}I$ is \casefive, $K^{\bb_{\ell}}I$ is \caseone.
Then the contribution to the sylvan matrix is given below.
$$%
\HH_0 K^{\aa} \!\otimes\! \langle \xx^{\aa} \rangle
\xleftarrow{
\monomialmatrix
  {i\\ j}
  {\begin{array}{@{}l@{\ }c@{\ }r@{}}
    \quad\, i j &\ \ i k &\, \, j k\ \,
    \\
      (-1)^{j,ij} \frac{\lceil \frac{3^r}{2} \rceil}{3^r 2^{s+1}} &(-1)^{i,ik} \frac{\lfloor \frac{3^r}{2} \rfloor}{3^r 2^{s+1}} &0\\
      (-1)^{i,ij} \frac{ \lceil \frac{3^r}{2} \rceil}{3^r 2^{s+1}}& (-1)^{k,ik} \frac{ \lfloor \frac{3^r}{2} \rfloor}{3^r 2^{s+1}}&0
   \end{array}}
  {\\\\}
\!\!\!}
{\HH_1 K^{\bb} \!\otimes\! \langle \xx^{\bb} \rangle}
$$
\end{lemma}

\begin{proof}
By Lemma \ref{zerotoone}, there are $\lceil \frac{3^r}{2} \rceil$ chain-link fences $\tau_r = ij$ with partial weight $1$. 
Since $\bb_r - \mathbf{c}_1 = e_j$, only these partial chain-link fences will produce chain-link fences that do not terminate.
At each degree $\mathbf{c}_{\ell}$, $T_1^{\mathbf{c}_{\ell}} = \{ ij \}$ and $S_0^{\mathbf{c}_{\ell}} = \{ i \}$ or $\{ j \}$.
Along the segment of $\lambda$ from $\mathbf{c}_1$ to $\mathbf{c}_s$, there is only one hedgerow that yields a chain-link fence of full length: the lattice path terminates unless $\mathbf{c}_{\ell-1} - \mathbf{c}_{\ell}$ is the standard basis vector corresponding to the vertex $\sigma_{\ell}$, and $\sigma_{\ell}$ must be the vertex in the stake set $S_0^{\mathbf{c}_{\ell}}$.
As in the proof of Lemma~\ref{zerototwotofivetothree}, the weights in this segment of the chain-link fence cancel.
Finally, if $\mathbf{c}_s - \aa = e_i$, $T_0^{\aa} = \{i \}$ or $\{ j \}$.
If $T_0^{\aa} = \{i\}$, then $\zeta_{T_0^{\aa}}(j) = j-i$, so $j$ is cycle-linked to $j$ with weight $1$ and to $i$ with weight $-1$.
Thus there are $\lceil \frac{3^r}{2} \rceil$ chain-link fences ending with terminal post $j$ and total weight $(-1)^{i,ij}$ and $\lceil \frac{3^r}{2} \rceil$ with terminal post $i$ and total weight $-(-1)^{i,ij} = (-1)^{j,ij}$.

Similarly, if $\mathbf{c}_s - \aa = e_j$, there are $\lceil \frac{3^r}{2} \rceil$ fences with terminal post $i$ and total weight $(-1)^{j,ij}$ and $\lceil \frac{3^r}{2} \rceil$ with terminal post $j$ and total weight $-(-1)^{j,ij} = (-1)^{i,ij}$.
Note that $\Delta_{1,\lambda}I = 3^r 2^s 2$.

For chain-link fences with initial post $ik$, 
there are $\lfloor \frac{3^r}{2} \rfloor$ chain-link fences with $\tau_r = ij$ with partial weight $(-1)^{i,ijk+1}$ by Lemma~\ref{zerotoone}.
Multiply by the partial weights described above.
\end{proof}

\begin{lemma}\label{zerotoonetotwotothree}
Suppose $\bb$ and $\aa$ are degree vectors such that $K^{\bb}I$ is \casezero and $K^{\aa}I$ is \casethree. Suppose also that $\lambda = (\aa, \mathbf{c}, \bb_r, \ldots, \bb_1, \bb)$ is a lattice path such that $K^{\mathbf{c}}I$ is \casetwo, $K^{\bb_{\ell}}I$ is \caseone.
Then the contribution to the sylvan matrix is given below.
$$%
\HH_0 K^{\aa} \!\otimes\! \langle \xx^{\aa} \rangle
\xleftarrow{
\monomialmatrix
  {i\\ j}
  {\begin{array}{@{}l@{\ }c@{\ }r@{}}
    \quad\, ij &\ \ ik &\, \, jk\ \,
    \\
      (-1)^{j,ij} \frac{\lceil \frac{3^r}{2} \rceil}{4 \cdot 3^r} &(-1)^{i,ik} \frac{\lfloor \frac{3^r}{2} \rfloor}{4\cdot 3^r}&0\\
      (-1)^{i,ij} \frac{\lceil \frac{3^r}{2} \rceil}{4\cdot 3^r} & (-1)^{k,ik} \frac{\lfloor \frac{3^r}{2} \rfloor}{4\cdot 3^r}&0 
   \end{array}}
  {\\\\}
\!\!\!}
{\HH_1 K^{\bb} \!\otimes\! \langle \xx^{\bb} \rangle}
$$
\end{lemma}

\begin{proof}
Again by Lemma~\ref{zerotoone}, there are $\lceil \frac{3^r}{2} \rceil$ chain-link fences with $\tau_r = ij$ and $\lfloor \frac{3^r}{2} \rfloor$ with $\tau_r = ik$.
After moving back in the $i$-direction to \casetwo, the chain-link fences with $\tau_r = ik$ will terminate, since $k$ cannot be a stake in degree $\mathbf{c}$.
The shrubbery $T_1^{\mathbf{c}} = \{ ij \}$, and the stake set $S_0^{\mathbf{c}} = \{ i \}$ or $\{ j \}$.
Following the proof of Lemma~\ref{zerototwotothree}, if $\mathbf{c} - \aa = e_i$, then there are in total $\lceil \frac{3^r}{2} \rceil$ chain-link fences with terminal post $j$ with weight $(-1)^{i,ij}$ and $\lceil \frac{3^r}{2} \rceil$ with terminal post $i$ with weight $-(-1)^{i,ij} = (-1)^{j,ij}$.

If $\mathbf{c} - \aa = e_j$, then there are in total $\lceil \frac{3^r}{2} \rceil$ with terminal post $i$ with weight $(-1)^{j,ij} = -(-1)^{i,ij}$ and $\lceil \frac{3^r}{2} \rceil$ with terminal post $j$ with weight $-(-1)^{j,ij} = (-1)^{i,ij}$.

Similarly, there are $\lfloor \frac{3^r}{2} \rfloor$ chain-link fences with initial post $ik$ and $\tau_r = ij$. 
These have partial weight $(-1)^{i,ijk+1}$ by Lemma \ref{zerotoone}.
Apply the arguments above to get the corresponding sylvan matrix entries.

Finally $\Delta_{1,\lambda} = 2 \cdot 2 \cdot 3^r$.
\end{proof}

\begin{lemma}\label{zerotoonetotwotofivetothree}
Suppose $\bb$ and $\aa$ are degree vectors such that $K^{\bb}I$ is \casezero and $K^{\aa}I$ is \casethree. Suppose also that $\lambda = (\aa, \aa_s, \ldots, \aa_1, \mathbf{c}, \bb_r, \ldots, \bb_1 \bb)$ is a lattice path such that $K^{\aa_{\ell}}I$ is \casefive, $K^{\mathbf{c}}I$ is \casetwo, $K^{\bb_{\ell}}I$ is \caseone.
Then the contribution to the sylvan matrix is given below.
$$%
\HH_0 K^{\aa} \!\otimes\! \langle \xx^{\aa} \rangle
\xleftarrow{
\monomialmatrix
  {i\\ j}
  {\begin{array}{@{}l@{\ }c@{\ }r@{}}
    \quad\, i j &\ \ i k &\, \, j k\ \,
    \\
      (-1)^{j,ij} \frac{\lceil \frac{3^r}{2} \rceil}{3^r \cdot 2^{s+2}} & (-1)^{i,ik} \frac{\lfloor \frac{3^r}{2} \rfloor}{3^r \cdot 2^{s+2}}&0\\
      (-1)^{i,ij} \frac{\lceil \frac{3^r}{2} \rceil}{^r \cdot 2^{s+2}} & (-1)^{k,ik} \frac{\lfloor \frac{3^r}{2} \rfloor}{3^r \cdot 2^{s+2}}&0 
   \end{array}}
  {\\\\}
\!\!\!}
{\HH_1 K^{\bb} \!\otimes\! \langle \xx^{\bb} \rangle}
$$
\end{lemma}

\begin{proof}
Segments of chain-link fences where the Koszul complex is \casefive are in bijection with lattice path segments and hedgerows.
In order for the chain-link fence not to terminate, the vertex $\sigma_{\ell}$ must be the vertex in the stake set $S_0$, and since $\sigma_{\ell} = \tau_{\ell-1} - e_{k_{\ell}}$, where $e_{k_{\ell}} = \bb_{\ell-1} - \bb_{\ell}$, this vertex is determined by $\lambda$. 
The weights of the linkages cancel, as in the proof of Proposition \ref{zerototwotofivetothree}.
Therefore the only difference between this case and the proof of Lemma~\ref{zerotoonetotwotothree} is $\Delta_{1, \lambda}I$.
Here, $\Delta_{1,\lambda}I = 3^r \cdot 2 \cdot 2^s \cdot 2$, where $r$ is the number of times the lattice path passes through \caseone and $s$ is the number of times the lattice path passes through \casefive.
\end{proof}




\begin{lemma}\label{zeroentrylemma}
Let $\lambda = (\aa = \bb_j, \bb_{j-1}, \ldots, \bb_1, \bb_0 = \bb)$ be a lattice path such that $K^{\bb}I$ is \casezero.
If an edge $e$ does not appear in any Koszul complex $K^{\bb_{\ell}}I$ except for $\ell = 0$, then $D_{v, e}^{\aa \bb} = 0$ for all vertices $v$.
\end{lemma}

\begin{proof}
If the initial post of a chain-link fence is the edge $ij$, then $\sigma_1 = j$.
If $ij$ is not an edge in $K^{\bb_1}I$, then $j$ is an isolated vertex of $K^{\bb_1}I$ and therefore is not a stake in $S_0^{\bb_1}$.
Thus the chain-link fence terminates.
\end{proof}

\section{Examples}

\begin{example}
Let $I = \langle xy, y^3, z \rangle$.
The staircase diagram and sylvan resolution for $I$ are given below.
To compute $D_{x,zy}^{111,131}$, note that there is only one lattice path from degree $131$ to $111$, each time moving back in the $y$-direction.
$K^{131}I$ has facets $zy, yx$, and $xz$, i.e., \casezero, $K^{121}I$ has facets $yx$ and $zy$, i.e., \caseone, and $K^{111}I$ has facets $yx$ and the vertex $z$, i.e., \casetwo.
Then by Theorem \ref{threevariabletheorem}(3), $D_{x,zy}^{111,131} = (-1)^{z,zy} \big( \frac{3^2 -1}{2 \cdot 3^2} \big) = \frac{4}{9}$.

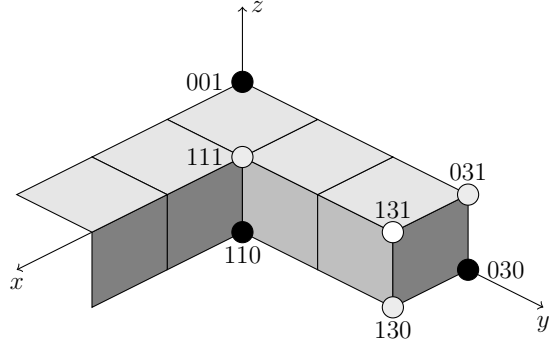
\begin{center}
\begin{tikzpicture}[scale=1,every node/.style={scale=0.8}]
      \draw[fill=gray] (0,0) -- (-1,-0.5) -- (-1,0.5) -- (0,1) -- (0,0);
      \draw[fill=gray] (-1,-0.5) -- (-2,-1) -- (-2,0) -- (-1,0.5) -- (-1,-0.5);
      \draw[fill=gray!50] (0,0) -- (1,-0.5) -- (1,0.5) -- (0,1) -- (0,0);
      \draw[fill=gray!50] (1,-0.5) -- (2,-1) -- (2,0) -- (1,0.5) -- (1,-0.5);
      \draw[fill=gray] (2,-1) -- (3,-0.5) -- (3,0.5) -- (2,0) -- (2,-1);
      \draw[fill=gray!20] (3,0.5) -- (2,0) -- (1,0.5) -- (2,1) -- (3,0.5);
      \draw[fill=gray!20] (1,0.5) -- (0,1) -- (1,1.5) -- (2,1) -- (1,0.5);
      \draw[fill=gray!20] (0,1) -- (-1,1.5) -- (0,2) -- (1,1.5) -- (0,1);
      \draw[fill=gray!20] (0,1) -- (-1,0.5) -- (-2,1) -- (-1,1.5) -- (0,1);
      \draw[fill=gray!20] (-1,0.5) -- (-2,0) -- (-3,0.5) -- (-2,1) -- (-1,0.5);
      \filldraw (0,0) circle (4pt);
      \node at (0,-0.3) {$110$};
      \filldraw[fill=gray!15] (0,1) circle (4pt);
      \node at (-0.5,1) {$111$};
      \filldraw (0,2) circle (4pt);
      \node at (-0.5,2) {$001$};
      \filldraw (3,-0.5) circle (4pt);
      \node at (3.5,-0.5) {$030$};
      \filldraw[fill=gray!15] (3,0.5) circle (4pt);
      \node at (3,0.8) {$031$};
      \filldraw[fill=white] (2,0) circle (4pt);
      \node at (2,0.3) {$131$};
      \filldraw[fill=gray!15] (2,-1) circle (4pt);
      \node at (2,-1.3) {$130$};
      \draw[->] (0,2) -- (0,3) node[right] {$z$};
       \draw[->] (-2,0) -- (-3,-0.5) node[below] {$x$};
       \draw[->] (3,-0.5) -- (4,-1) node[below] {$y$};
\end{tikzpicture}
\end{center}

{\foot
$$%
\begin{array}{@{}c@{\,}}
   \HH_{-1} K^{110} \!\otimes\! \langle xy\rangle 
\\ \oplus
\\ \HH_{-1} K^{030} \!\otimes\! \langle y^3\rangle 
\\ \oplus
\\ \HH_{-1} K^{001} \!\otimes\! \langle z\rangle 
\end{array}
\xleftarrow{\!\!\!
\monomialmatrix
  {\nothing\\\nothing\\\nothing}
  {\begin{array}{@{\,}c@{\ }c@{\ }c@{\,}}
       x\ y\ z    & x\ y    & y\ z   
  \\{}\![\,0\ 0\ 1\,] &  [\,0\ 1\,] &  [\,0\ 0\,]\!
  \\{}\![\,0\ 0\ 0\,] &  [\,1\ 0\,] &  [\,0\ 1\,]\!
  \\{}\![\,1\ 1\ 0\,] &  [\,0\ 0\,] &  [\,1\ 0\,]\!
  \end{array}}
  {\\\\\\}
\!\!\!}
{\begin{array}{@{}c@{\,}}
   \HH_0 K^{111} \!\otimes\! \langle xyz\rangle 
\\ \oplus
\\ \HH_0 K^{130} \!\otimes\! \langle xy^3\rangle 
\\ \oplus
\\ \HH_0 K^{031} \!\otimes\! \langle y^3z\rangle 
\end{array}
}
\xleftarrow{\!\!\!
\monomialmatrix
  { x\\ y\\ z\\[.2ex] x\\ y\\[.2ex] y\\ z}
  {\begin{array}{@{}l@{\ }c@{\ }r@{}}
    \quad\, zy &\ \ yx & xz\ \,
    \\
    \!
    \left[
    \begin{array}{@{}r@{}}
	   \frac{4}{9}
	\\ \frac{1}{9}
	\\ -\frac{5}{9}
    \end{array}
    \right.
    &
    \begin{array}{@{}r@{}}
	   \frac{5}{9}
	\\ -\frac{1}{9}
	\\ -\frac{4}{9}
    \end{array}
    &
    \left.
    \begin{array}{@{}c@{}}
	  \ \ 0\
	\\\ \ 0\
	\\\ \ 0\
    \end{array}
    \right]
    \!
    \\[3ex]
    \hspace{-.1ex}
    \left[
    \begin{array}{@{}r@{}}
	  -\frac{1}{2}\!\!
	\\\frac{1}{2}\!\!
    \end{array}
    \right.
    &
    \begin{array}{@{\!\!}r@{}}
	  \ \ 0
	\\\ \ 0
    \end{array}
    &
    \left.
    \begin{array}{@{}r@{}}
	   \!\!-\frac{1}{2}\,
	\\ \!\!\frac{1}{2}\,
    \end{array}
    \right]
    \hspace{-.1ex}
    \\
    \hspace{-.1ex}
    \left[
    \begin{array}{@{}r@{}}
	  \ \ \ 0
	\\\ \ \ 0
    \end{array}
    \right.
    &
    \begin{array}{@{}r@{}}
	   -\frac{1}{2}\
	\\ \frac{1}{2}\
    \end{array}
    &
    \left.
    \begin{array}{@{}r@{}}
	   \!\!-\frac{1}{2}\,
	\\ \!\!\frac{1}{2}\,
    \end{array}
    \right]
    \hspace{-.1ex}
    \end{array}}
  {\\\\\\[.6ex]\\\\[.6ex]\\\\}
\!\!\!}
{ \HH_1 K^{131} \!\otimes\! \langle xy^3z\rangle }
$$
}%
\end{example}

\begin{example}
Let $I = \langle yz,xz,xy^2,x^2y \rangle$, whose staircase diagram is given below.

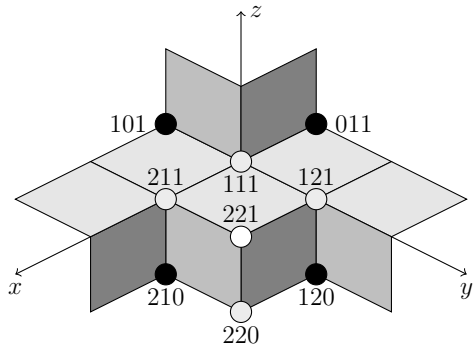
\begin{center}
\begin{tikzpicture}[scale=1,every node/.style={scale=0.8}]
      \draw[fill=gray] (0,0) -- (1,0.5) -- (1,1.5) -- (0,1) -- (0,0);
      \draw[fill=gray!50] (0,0) -- (-1,0.5) -- (-1,1.5) -- (0,1) -- (0,0);
      \draw[fill=gray] (-1,0.5) -- (-2,0) -- (-2,1)  -- (-1,1.5) -- (-1,0.5);
      \draw[fill=gray!20] (-1,1.5) -- (-2,2) -- (-1,2.5) -- (0,2) -- (-1,1.5);
      \draw[fill=gray!20] (-1,1.5) -- (-2,1) -- (-3,1.5) -- (-2,2) -- (-1,1.5);
      \draw[fill=gray!20] (-1,1.5) -- (0,1) -- (1,1.5) -- (0,2) -- (-1,1.5);
      \draw[fill=gray!50] (0,2) -- (0,3) -- (-1,3.5) -- (-1,2.5) -- (0,2);
      \draw[fill=gray] (0,2) -- (1,2.5) -- (1,3.5) -- (0,3) -- (0,2);
      \draw[fill=gray!20] (0,2) -- (1,2.5) -- (2,2) -- (1,1.5) -- (0,2);
      \draw[fill=gray!20] (1,1.5) -- (2,2) -- (3,1.5) -- (2,1) -- (1,1.5);
      \draw[fill=gray!50] (1,1.5) -- (2,1) -- (2,0) -- (1,0.5) -- (1,1.5);
      \filldraw (1,2.5) circle (4pt);
      \node at (1.5,2.5) {$011$};
      \filldraw (-1,2.5) circle (4pt);
      \node at (-1.5,2.5) {$101$};
      \filldraw (-1,0.5) circle (4pt);
      \node at (-1,0.2) {$210$};
      \filldraw (1,0.5) circle (4pt);
      \node at (1,0.2) {$120$};
      \filldraw[fill=gray!15] (-1,1.5) circle (4pt);
      \node at (-1,1.8) {$211$};
      \filldraw[fill=gray!15] (1,1.5) circle (4pt);
      \node at (1,1.8) {$121$};
      \filldraw[fill=gray!15] (0,2) circle (4pt);
      \node at (0,1.7) {$111$};
      \filldraw[fill=white] (0,1) circle (4pt);
      \node at (0,1.3) {$221$};
      \filldraw[fill=gray!15] (0,0) circle (4pt);
      \node at (0,-0.3) {$220$};
      \draw[->] (0,3) -- (0,4) node[right] {$z$};
       \draw[->] (-2,1) -- (-3,0.5) node[below] {$x$};
       \draw[->] (2,1) -- (3,0.5) node[below] {$y$};
\end{tikzpicture}
\end{center}

To compute $D_{x,zy}^{220,221}$, note that there is one lattice path from $221$ to $220$, moving back in the $z$-direction. 
$K^{221}I$ is \casezero, and $K^{220}I$ has facets $x$ and $y$, i.e., \casethreeprime.
By Theorem \ref{threevariabletheorem}(2), $D_{x,zy}^{220,221} = (-1)^{y,zy}\big( \frac{1}{2} \big) = -\frac{1}{2}$.
The sylvan resolution for $I$ is given below.

{\foot
$$
\begin{array}{@{}c@{\!\!\!}}
   \HH_{{\!-\!}1} K^{011} \!\otimes\! \langle yz \rangle \quad
\\ \oplus
\\ \HH_{{\!-\!}1} K^{101} \!\otimes\! \langle xz \rangle \quad
\\ \oplus
\\ \HH_{{\!-\!}1} K^{120} \!\otimes\! \langle xy^2 \rangle \ \,
\\ \oplus
\\ \HH_{{\!-\!}1} K^{210} \!\otimes\! \langle x^2y \rangle \ \,
\end{array}
\xleftarrow{\!\!\!
\monomialmatrix
  {\nothing\!\\\nothing\!\\\nothing\!\\\nothing\!}
  {\scriptstyle
   \begin{array}{@{\!}c@{\ }c@{\hspace{-.3ex}}*2{c@{\ }c@{\ }c@{\hspace{-.3ex}}}c@{\ }c@{\!}}
       x\!\!&\!\!y
     & x\!\!&\!y&\!\!z
     & x\!\!&\!y&\!\!z
     & x\!\!&\!\!y
  \\[-.2ex][\,0&0\,]&[\,\frac{3}{4}&\frac{3}{4}&0\,]&[\,\frac{1}{4}&\frac{1}{4}&0\,]&[\,1&0\,]
  \\[-.2ex][\,0&0\,]&[\,\frac{1}{4}&\frac{1}{4}&0\,]&[\,\frac{3}{4}&\frac{3}{4}&0\,]&[\,0&1\,]
  \\[-.2ex][\,1&0\,]&[\,\,0\,& 0\, &1\,]&[\,\,0\,& 0\, &0\,]&[\,0&0\,]
  \\[-.2ex][\,0&1\,]&[\,\,0\,& 0\, &0\,]&[\,\,0\,& 0\, &1\,]&[\,0&0\,]
  \end{array}}
  {\\\\\\\\}
\!\!\!}
{\begin{array}{@{}c@{\,}}
   \HH_{\!0} K^{220} \!\otimes\! \langle x^2y^2 \rangle 
\\ \oplus
\\ \HH_{\!0} K^{121} \!\otimes\! \langle xy^2z \rangle 
\\ \oplus
\\ \HH_{\!0} K^{211} \!\otimes\! \langle x^2yz \rangle 
\\ \oplus
\\ \HH_{\!0} K^{111} \!\otimes\! \langle xyz \rangle 
\end{array}
}
\xleftarrow{\!\!\!
\monomialmatrix
  {x\\y\\[.2ex]
    x\\ y\\ z\\[.2ex]
    x\\ y\\ z\\[.2ex]
    x\\ y}
  {\begin{array}{@{\!}l@{}c@{}r@{\!}}
    \ \ zy &\ yx & xz\ \,
    \\
    \hspace{-.2ex}
    \left[
    \begin{array}{@{}r@{}}
	  -\frac{1}{2}
	\\ \frac{1}{2}
    \end{array}
    \right.
    &
    \begin{array}{@{}r@{}}
	  \ \ 0\
	\\\ \ 0\
    \end{array}
    &
    \left.
    \begin{array}{@{}r@{}}
	   -\frac{1}{2}\,
	\\ \frac{1}{2}\,
    \end{array}
    \right]
    \hspace{-.2ex}
    \\[1.5ex]
    \!
    \left[\!
    \begin{array}{@{}r@{}}
	  \ \ 0
	\\\ \ 0
	\\\ \ 0
    \end{array}
    \right.
    &
    \begin{array}{@{}r@{}}
	   \frac{1}{3}
	\\ -\frac{2}{3}
	\\ \frac{1}{3}
    \end{array}
    &
    \left.
    \begin{array}{@{}r@{}}
	   -\frac{1}{3}
	\\ -\frac{1}{3}
	\\ \frac{2}{3}
    \end{array}
    \right]
    \!
    \\[3ex]
    \!
    \left[\!
    \begin{array}{@{}r@{}}
	   \frac{1}{3}
	\\ \frac{1}{3}
	\\ -\frac{2}{3}
    \end{array}
    \right.
    &
    \begin{array}{@{}r@{}}
	   \frac{2}{3}
	\\ -\frac{1}{3}
	\\ -\frac{1}{3}
    \end{array}
    &
    \left.
    \begin{array}{@{}r@{}}
	   0\,
	\\ 0\,
	\\ 0\,
    \end{array}
    \right]
    \!
    \\[3ex]
    \hspace{-.2ex}
    \left[\!
    \begin{array}{@{}r@{}}
	  \ \ 0
	\\\ \ 0
    \end{array}
    \right.
    &
    \begin{array}{@{}r@{}}
	   \frac{1}{2}
	\\ -\frac{1}{2}
    \end{array}
    &
    \left.
    \begin{array}{@{}r@{}}
	   0\,
	\\ 0\,
    \end{array}
    \right]
    \hspace{-.2ex}
    \end{array}}
  {\\\\[.6ex]\\\\\\[.6ex]\\\\\\[.6ex]\\\\}
\!\!\!}
{ \widetilde{H}_{\!1\hspace{-.1ex}} K^{221} \!\otimes\! \langle x^2y^2z \rangle }
$$
}
\end{example}


\begin{thebibliography}{EMO20}
\bibitem[Ber86]{berg86}
Lothar Berg, \emph{Three results in connection with inverse matrices}, Proceedings from the symposium on operator theory, Linear Algebra Appl., \textbf{84} (1986), 63--77.

\bibitem[CCK17]{cck17}
Michael J. Catanzaro, Vladimir Y. Chernyak, and John R. Klein, \emph{A higher Boltzmann distribution}, Journal of Applied Computational Topology \textbf{1} (2017), no. 2, 215--240.

\bibitem[DKM09]{dkm09}
Art M. Duval, Caroline J. Klivans, and Jeremy L. Martin, \emph{Simplicial matrix-tree theorems}, Transactions of the American Mathematical Society \textbf{361} (2009), 607--611.

\bibitem[Eag90]{eag90}
John Eagon, \emph{Partially split double complexes with an associated Wall complex and applications to ideals generated by monomials}, Journal of Algebra \textbf{135} (1990), no. 2, 344--362.

 \bibitem[EMO21]{emo20}
 John Eagon, Ezra Miller, and Erika Ordog, \emph{Minimal resolutions of monomial ideals} (submitted), 2021. arXiv:1906.08837v2.
 
\bibitem[GM88]{gm88}
R\"udiger Gebauer and H. Michael M\"oller, \emph{On an installation of Buchberger's algorithm}, J. Symbolic Comput. \textbf{6} (1988), no. 2--3, 275--286.
 
 \bibitem[Hoc77]{hochster77}
 Melvin Hochster, \emph{Cohen--Macaulay rings, combinatorics, and simplicial complexes}, Ring Theory, II, Lecture notes in pure and applied mathematics Vol. 26, Marcel Decker, New York, 1977, 171--223.
 
 \bibitem[Mil02]{miller-Planar2002}
Ezra Miller, \emph{Planar graphs as minimal resolutions of trivariate
  monomial ideals}, Documenta Math.~\textbf{7}~(2002), 43--90.
  
  \bibitem[MS05]{cca}
Ezra Miller and Bernd Sturmfels, \emph{Combinatorial commutative
  algebra}, Graduate Texts in Mathematics, vol.~227, Springer-Verlag,
  New York, 2005.
  
  \bibitem[Pai15]{painter15}
Jared L. Painter, \emph{On the behavior of minimal free resolutions of trivariate generic monomial ideals}, Communications in Algebra, \textbf{43} (2015), no. 2, 521--540.
\end{thebibliography}
\end{document}